%% file: FVBreakAggregation.tex
\documentclass[a4paper,11pt]{article}
\usepackage{amsmath}
\usepackage{amsfonts}
\usepackage[latin9]{inputenc}
\usepackage[english]{varioref}
\usepackage{dcolumn}
\usepackage[height=22cm , width = 16cm , top = 4cm , left = 3cm, a4paper]{geometry}
\usepackage[a4paper]{geometry}
\usepackage[final]{graphicx}
\usepackage{epsfig}
\usepackage{pstricks}
\usepackage{psfrag}
\usepackage{rotating}
\usepackage{booktabs}
\usepackage{delarray}
\usepackage{amssymb}
\usepackage{rotating}
\usepackage{subfigure}
\usepackage{layout}
\usepackage{amsthm}
\usepackage{hyperref}

{\newcommand{\dd}{\,d}                    
{\newcommand{\pd}{\partial}                    
{\newcommand{\pp}{.}                      
{\newcommand{\pk}{,}

\newtheorem{thm}{Theorem}[section]

\newtheorem{lem}[thm]{Lemma}

\newtheorem{rem}[thm]{Remark}
\newtheorem{defin}[thm]{Definition}

\parindent 0pt
\newcommand{\enter}{\bigskip}

\newcommand{\dps}{\displaystyle}

\begin{document}
\thispagestyle{empty}

\author{{Rajesh Kumar$$\footnote{Corresponding author: currently working at RICAM, Austrian Academy of Sciences, Altenberger
strasse 69, 4040 Linz, Austria.
Tel: +43-73224685248, {\it{${}$ Email address: }}rajineurope@gmail.com} { }, Jitendra Kumar and Gerald Warnecke}\vspace{.2cm}\\
\small \it Institute for Analysis and Numerics, Otto-von-Guericke University Magdeburg,\\
\small \it Universit\"{a}tsplatz 2, D-39106 Magdeburg, Germany \vspace{.2cm} \\
}

\title{Convergence analysis of a finite volume scheme for solving non-linear aggregation-breakage population balance equations}

\maketitle

%%%%%%%%%%%%%%%%%%%%%%%%%%%% Einleitung %%%%%%%%%%%%%%%%%
\hrule \vskip 8pt

\begin{quote}
{\small {\em\bf Abstract}

This paper presents stability and convergence analysis of a finite volume scheme (FVS) for solving aggregation, breakage 
and the combined processes by showing Lipschitz continuity of the numerical fluxes. It is shown that the FVS is second 
order convergent independently of the meshes for pure breakage problem while for pure aggregation and coupled 
equations, it shows second order convergent on uniform and non-uniform smooth meshes. Furthermore, it gives only 
first order convergence on non-uniform grids. The mathematical results of convergence analysis are also 
demonstrated numerically for several test problems.}
\end{quote}


%{\bf AMS Subject Classifications: } 45J05, 65R20, 45L05.

\vskip 10pt \hrule

\section{Introduction}
The aggregation-breakage population balance equations (PBEs) are the models for the growth of particles by combined 
effect of aggregation and breakage. Each particle is identified here by its size, i.e.\ volume or mass. The 
equations we consider in this paper describe the time evolution of the particle size distribution (PSD) under the 
simultaneous effect of binary aggregation and multiple breakage. In binary aggregation, two particles combine together 
to form a bigger one whereas in breakage process, a big particle breaks into two or many fragments. There are many
engineering applications, including aerosol physics, high shear granulation, highly demanding nano-particles and 
pharmaceutical industries etc., see Sommer et al.\ \cite{Sommer:2006}, Gokhale 
et al.\ \cite{Gokhale:2009} and references therein. Binary breakage is not sufficient for some of these applications, 
therefore, multiple fragmentation is preferred. The temporal change
of the particle number density, $f(t,x) \geq 0$, of particles of volume
$x \in \mathbb{R}_{>0}$ at time $t\in \mathbb{R}_{>0}$ in a
spatially homogeneous physical system undergoing an aggregation-breakage process is
described by the following well known PBEs,
see \cite{Melzak:1957, Ziff:1991}
\begin{align}\label{pbe1}
\frac{\pd f(t, x)}{\pd t}= &\frac{1}{2} \int_{0}^{x}\beta(x-u,u) f(t, x-u) f(t, u) du -
\int_{0}^{\infty} \beta(x, u) f(t,u) f(t, x) du \nonumber \\ & +\int_x^\infty b(x, \epsilon) S(\epsilon) f(t, \epsilon) \dd \epsilon -S(x) f(t, x) \pk
\end{align}
with initial data
\begin{align}\label{pbe1_IC}
f(0,x) = f^{\text{in}}(x)\geq 0,\ \ \ x\in ]0,\infty[.
\end{align}
The first two terms on the right-hand side (rhs) are due to aggregation while the third and fourth terms
model the breakage process. The two positive terms describe the
creation of particles of size $x$ and are called the birth terms for
aggregation respectively breakage. The two negative terms describe
the disappearance of particles of size $x$ and are commonly called the
death terms. The aggregation kernel $\beta(x,y)\geq 0$ characterizes the rate at which two particles of volumes $x$ 
and $y$ combine together. It also satisfies the symmetry condition $\beta(x, y) = \beta(y, x)$. The selection function $S(\epsilon)$ describes the rate
at which particles of size $\epsilon$ are selected to break. The
breakage function $b(x, \epsilon)$ for a given $\epsilon>0$ gives the size distribution of particle sizes 
$x\in [0,\epsilon[$ resulting from the breakage of a particle of size $\epsilon$. For the particular case 
of $b(x,\epsilon)= 2/\epsilon$, the multiple breakage PBE turns into the binary breakage PBE. The breakage function 
has the following important properties
\begin{align}\label{breakageproperties}
\int_0^x b(u,x)du= \bar{N}(x), \quad \int_0^x u b(u,x)du= x.
\end{align}
The function $\bar{N}(x)$, which may be infinite, denotes the number of fragments obtained from the breakage of particle of size $x$. The second integral shows that the total mass created from the breakage of a particle of size $x$ is again $x$. In aggregation-breakage processes the total number of particles varies in time while the total mass of particles remains conserved. In terms of $f$, the total number of particles and the total mass of particles at time $t\geq 0$ are respectively given by
\begin{align*}
M_0(t):= \int_0^\infty f(t,x)dx, \ \ \  M_1(t):= \int_0^\infty x f(t,x)dx.
\end{align*}
It is easy to show that the total number of particles $M_0(t)$ decreases by aggregation and increases by breakage processes while the total mass $M_1(t)$ does not vary during these events.
For the total mass conservation
\begin{align*}
\int_0^\infty x f(t,x) \dd x = \int_0^\infty x f^{in}(x) \dd x, \ \
t\geq 0 \pk
\end{align*}
holds. However, for some special cases of $\beta$ when it is sufficiently large compared to the selection function $S$, a phenomenon called gelation occurs. In this case the total mass of particles is not conserved, see Escobedo et al. \cite{Escobedo:2003} and further citations for details. \enter

Mathematical results on existence and uniqueness of solutions of
the equation (\ref{pbe1}) and further citations can be found in McLaughlin et al.\ \cite{Mclaughlin:1998} and W. Lamb \cite{Lamb:2004} for rather
general aggregation kernels, breakage and selection functions. In our analysis
we consider them to be twice continuously differentiable functions.
The PBEs (\ref{pbe1}) can only be solved analytically for a limited number of simplified problems, see Ziff \cite{Ziff:1991}, Dubovskii et al.\ \cite{Dubovskii:1992} and the references therein.
This certainly leads to the necessity of using numerical methods for
solving general PBEs. Several numerical methods have been introduced to solve the PBEs. Stochastic methods (Monte-Carlo) have been developed, see Lee and Matsoukas \cite{KLee:2000} for solving equations of aggregation with binary breakage. Finite element techniques can be found in Mahoney and Ramkrishna \cite{Mahoney:2002} and the references therein for the equations of simultaneous aggregation, growth and nucleation. Some other numerical techniques are available in the literature such as the method of successive
approximations by D. Ramkrishna \cite{Ramkrishna:2000}, method of moments \cite{Madras:2004,Marchisio:2005}, 
finite volume methods \cite{Motz:2002,Rajesh:M3AS} and sectional methods \cite{J_Kumar:2006_Thesis,Kumar:1996-1,Vanni:2002} to solve such PBEs.\enter

A completely different numerical approach was proposed by Filbet and
Lauren\c{c}ot \cite{Filbet:2004} for solving aggregation PBEs by discretizing a well known mass balance formulation. 
They thereby
introduced an application of the FVS to solve the aggregation problem. Further, Bourgade and Filbet
\cite{Filbet:2007} have extended their scheme to solve the case of binary
aggregation and binary breakage PBEs and gave a convergence proof of approximate solutions in the space 
$L^\infty(0,T;L^1(0,\mathrm{R}))$. For a special case of a uniform mesh they have shown error estimates of first order. The scheme has also been extended to two-dimensional aggregation problems
by Qamar and Warnecke \cite{Qamar:2007}. Finally it has been
observed that the FVS is a good alternative to the methods mentioned above for solving the PBEs due to its automatic mass
conservation property.\enter

Since Bourgade and Filbet have considered aggregation with binary breakage problems on uniform meshes only. 
The objective here is to analyze such a FVS to solve the aggregation with multiple breakage PBEs on general meshes.
We also demonstrate mathematically the missing stability and the convergence analysis of the FVS for simultaneous 
aggregation-breakage PBEs by following Hundsdorfer and Verwer \cite{Hundsdorfer:2003} and Linz \cite{Linz:1975}. 
The mathematical results are verified numerically for several test problems on four different types of uniform and non-uniform grids. \enter

This paper is organized as follows. First, we derive the FVS to solve aggregation-breakage PBEs. Then in 
Section \ref{convergence:section3} some useful definitions and theorems are reviewed from 
\cite{Hundsdorfer:2003, Linz:1975} which are used in further analysis of the method. 
Here we also discuss the consistency and prove the Lipschitz continuity of the numerical fluxes to 
get the convergence results. Later on the convergence analysis is numerically tested for several 
problems in Section \ref{num:resultsfvsaggbrk}. Further, Section \ref{conclusions:fvsaggbrk} summarizes 
some conclusions. At the end of the paper one Appendix is provided which gives a bound on total number of particles for the aggregation-breakage terms.

\section{Finite volume scheme}\label{S:PureBreakFV}
In this section a FVS for solving aggregation-breakage PBEs is discussed. Following Filbet and 
Lauren\c{c}ot \cite{Filbet:2004} for aggregation, a new form
of the breakage PBE is presented in order to apply the
FVS efficiently. Then stability and convergence analysis will be
discussed for the method.
\subsection{Aggregation-breakage PBE in a conservative form}
Writing the aggregation and breakage terms in divergence form enable us to get a precise amount of mass dissipation or conservation. It can be written in a conservative form of mass density $xf(t,x)$ as
\begin{align}\label{E:FormCL}
\frac{\pd \left[x f(t,x)\right]}{\pd t} + \frac{\pd}{\pd x}
\bigg(F^{\text{agg}}(t,x) + F^{\text{brk}}(t,x)\bigg) = 0.
\end{align}
The abbreviations \emph{agg} and \emph{brk} are used for aggregation
and breakage terms respectively. The flux functions $F^{\text{agg}}$
and $F^{\text{brk}}$ are given by
\begin{align}\label{I:Aggfluxexact}
F^{\text{agg}}(t,x) =  \int_{0}^{x}\int_{x-u}^{\infty} u \beta(u,v)
f(t, u) f(t, v) dv du,\quad \text{and}
\end{align}

\begin{align}\label{I:Brkfluxexact}
F^{\text{brk}}(t,x) =  -\int_{x}^{\infty}\int_{0}^{x} u b(u,v) S(v)
f(t, v)du dv.
\end{align}
It should be noted that both forms of aggregation-breakage
PBEs (\ref{pbe1}) and (\ref{E:FormCL}) are interchangeable by using the Leibniz integration rule. The concept of this conservative formulation of the PBE has been used in Tanaka et al.\ \cite{Tanaka:1996} and Makino et al.\ \cite{Makino:1998}. It should also be mentioned that the equation (\ref{E:FormCL}) reduces into the case of pure aggregation or pure breakage process when $F^{\text{brk}}(t,x)$ or $F^{\text{agg}}(t,x)$ is zero, respectively. \enter

In the PBE (\ref{E:FormCL}) the volume
variable $x$ ranges from $0$ to $\infty$. In order to apply a
numerical scheme for the solution of the equation a first step is to
fix a finite computational domain $\Omega:= ]0,x_{\text{max}}]$ for an $0< x_{\text{max}}<\infty$. Hence, for $x\in \Omega$ and time $t\in (0,T]$ where $T<\infty$, the aggregation and the breakage fluxes for the truncated conservation law for $n$, i.e.\ for
\begin{align}\label{E:FormCLtrun}
\frac{\pd \left[x n(t,x)\right]}{\pd t} + \frac{\pd}{\pd x}
\bigg(F^{\text{agg}}(t,x) + F^{\text{brk}}(t,x)\bigg) = 0
\end{align}
are given as
\begin{align}\label{I:Aggfluxexacttrun}
F^{\text{agg}}(t,x) =  \int_{0}^{x}\int_{x-u}^{x_{\text{max}}} u \beta(u,v)
n(t, u) n(t, v) dv du, \quad \text{and}
\end{align}

\begin{align}\label{I:Brkfluxexacttrun}
F^{\text{brk}}(t,x) =  -\int_{x}^{x_{\text{max}}}\int_{0}^{x} u b(u,v) S(v)
n(t, v)du dv.
\end{align}
Here the variable $n(t,x)$ denotes the solution to the truncated equation. We are given with initial data
\begin{align}\label{E:IC}
n(0,x) = f^{\text{in}}(x), \quad x\in \Omega.
\end{align}
For further analysis, all the kinetic parameters $\beta$, $S$ and $b$ are considered to be two times continuously differentiable function, i.e.\
\begin{align}\label{kernelcondition}
\beta, b \in \mathcal{C}^2(]0, x_{\text{max}}]\times ]0, x_{\text{max}}]) \ \text{and}\ \ S\in \mathcal{C}^2(]0, x_{\text{max}}]).
\end{align}
From (\ref{kernelcondition}), there exists some non-negative constants $Q$ and $Q_1$ depending on $x_{\text{max}}$ such that
\begin{align}\label{kernelbound}
\beta(x,y)\leq Q \quad \text{and} \quad b(x,y)S(y)\leq Q_1 \quad \text{for}\quad x,y\in ]0,x_{\text{max}}].
\end{align}

\begin{rem} The formulation we use here is a non-conservative truncation for the pure aggregation operator as
$F^{\text{agg}}(t,x_{\text{max}})\geq 0$ while it is mass conserving for the pure breakage equation, i.e.\ $F^{\text{brk}}(t,x_{\text{max}})=0$.
Hence, the combined formulation (\ref{E:FormCLtrun}) is a non-conservative truncation as used by Bourgade and Filbet \cite{Filbet:2007}. One could make a conservative truncation by replacing $x_{\text{max}}$ by $x_{\text{max}}-u$ in (\ref{I:Aggfluxexacttrun}). This would give $F^{\text{agg}}(t,x_{\text{max}})= 0$. But it describes an artificial interruption of the aggregation process without a real physical justification. With our truncation particles that are too large leave the system. \enter
\end{rem}

\subsection{Numerical discretization} Finite volume methods are a class of discretization
schemes used to solve mainly conservation laws, see LeVeque
\cite{Leveque:2002}. For a semi-discrete scheme, the interval $]0,x_{\text{max}}]$ is discretized into small cells
\begin{align*}
\Lambda_i := ]x_{i-1/2},x_{i+1/2}], \quad \ i= 1,...,I  \pk \quad \text{with}
\end{align*}
\begin{align*}
x_{1/2} = 0, \quad x_{I+1/2}= x_{\text{max}}, \quad \Delta x_i = x_{i+1/2} -
x_{i-1/2} \leq \Delta x,
\end{align*}
where $\Delta x$ is the maximum mesh size. The representative of each size, usually the center of each cell $x_i = (x_{i-1/2}+x_{i+1/2})/2$, is called pivot or grid point.
The FVS has been carried over to the discretization of such equations by instead
of interpreting $\hat{n}_i(t)$ as an approximation to a point value at a grid point, i.e.\ $ n(t, x_i)$, rather taking an
approximation of the cell average of the solution on cell $i$ at
time $t$
\begin{align}\label{numdenapprox}
\hat{n}_i(t)\approx n_i= \frac{1}{\Delta x_i}\int_{x_{i-1/2}}^{x_{i+1/2}}n(t,x){\rm d}x \pp
\end{align}
Integrating the conservation law on a cell in space $\Lambda_i$, the FVS is given as
\cite{Leveque:2002}
\begin{align}\label{E:GenDis}
\frac{x_i d \hat{n}_i(t)}{dt} = -\frac{1}{\Delta x_i}
\bigg[J^{\text{agg}}_{i+1/2}-J^{\text{agg}}_{i-1/2} + J^{\text{brk}}_{i+1/2} - J^{\text{brk}}_{i-1/2}\bigg].
\end{align}
The term $J^{-}_{i+1/2}$ is called \emph{the numerical flux} which is an
appropriate approximation of the truncated continuous flux function $F^{\text{agg}}$ and/or $F^{\text{brk}}$ depending upon the processes under consideration.

In case of a breakage process, the numerical flux may be
approximated from the mass flux $F^{\text{brk}}$ as follows
\begin{align} \label{E:NumericalFluxApproximation}
F^{\text{brk}}(x_{i+1/2}) & =  -\int_{x_{i+1/2}}^{x_{\text{max}}} \int_0^{x_{i+1/2}} u b(u, \epsilon) S(\epsilon)
n(t, \epsilon) \dd u \dd \epsilon \nonumber \\
& =  -\sum_{k=i+1}^{I} \int_{\Lambda_k} S(\epsilon) n(t, \epsilon) \sum_{j=1}^i \int_{\Lambda_j} u b(u, \epsilon) \dd u \dd \epsilon.
\end{align}
Using our assumptions that $S\in \mathcal{C}^2( ]0,x_{\text{max}}])$, $b\in \mathcal{C}^2( ]0,x_{\text{max}} ]\times ]0,x_{\text{max}} ])$ and applying the mid point rule we can rewrite (\ref{E:NumericalFluxApproximation}) as
\begin{align}\label{2}
F^{\text{brk}}(x_{i+1/2}) &= \underbrace{-\sum_{k=i+1}^{I} n_k(t) S(x_k)\Delta x_k \sum_{j=1}^i x_j b(x_j,x_k)\Delta x_j}_{=:J^{\text{brk}}_{i+1/2}(n)} + {\cal O}(\Delta x^2)
\end{align}
Similarly for the aggregation problem,
\begin{align}\label{aggfluxmon26}
F^{\text{agg}}(x_{i+1/2}) = \int_{0}^{x_{i+1/2}}\int_{x_{i+1/2}-u}^{x_{\text{max}}} u \beta(u,v)
n(t, u) n(t, v) dv du.
\end{align}
From Filbet and Lauren\c{c}ot \cite{Filbet:2004}, the above equation can be written as
\begin{align*}
F^{\text{agg}}(x_{i+1/2}) = \sum_{k=1}^{i} (xn)_k \Delta x_k \Bigg(\sum_{j=\alpha_{i,k}}^{I} (xn)_j \int_{\Lambda_j} \frac{\beta(x,x_k)}{x}dx
+& (xn)_{\alpha_{i,k}-1} \int_{x_{i+1/2}-x_k}^{x_{\alpha_{i,k}-1/2}}\frac{\beta(x,x_k)}{x}dx \Bigg)\\&+{\cal O}(\Delta x^2).
\end{align*}
Here, the parameter $I$ denotes the number of cells. The integer
$\alpha_{i,k}$ corresponds to the index of each cell such that
\begin{align}\label{indexagg}
x_{i+1/2} - x_k \in \Lambda_{\alpha_{i,k}-1}.
\end{align}
Applying mid point approximation for the first term and Taylor series expansion of the second term about the point $x_{\alpha_{i,k}-1}$ give with $(xn)_k= x_kn_k$
\begin{align}\label{aggnumflux}
F^{\text{agg}}(x_{i+1/2}) =& \underbrace{\sum_{k=1}^{i} x_k n_k \Delta x_k \Bigg( \sum_{j=\alpha_{i,k}}^{I}n_j
\beta_{j,k} \Delta x_j +
n_{\alpha_{i,k}-1} \beta_{\alpha_{i,k}-1,k}(x_{\alpha_{i,k}-1/2} - (x_{i+1/2}-x_k))\Bigg)}_{=:J^{\text{agg}}_{i+1/2}(n)}\nonumber\\&+{\cal O}(\Delta x^2)
\end{align}

Let us denote the vector ${\bf{{n}}}:= [{n}_1,\ldots,{n}_I]$ obtained by $L^2$ projection of the exact solution $n$ 
into the space of step functions constant on each cell. It is worth to mention that this projection error can easily be shown of second 
order, see remark 3.3.3 in \cite{Rajesh:thesis}. We also define the vectors $$\Delta {\bf{J}}^{\text{agg}}({\bf{n}}):= [\Delta J_{1}^{\text{agg}}({\bf{n}}),\ldots,\Delta J_{I}^{\text{agg}}({\bf{n}})]\quad \ \text{and}\quad \ \Delta {\bf{J}}^{\text{brk}}({\bf{n}}):= [\Delta {J}_{1}^{\text{brk}}({\bf{n}}),\ldots, \Delta{J}_{I}^{\text{brk}}({\bf{n}})]$$
where
\begin{align}\label{naya1}
\Delta J_{i}^{\text{agg}}({\bf{n}})=\frac{1}{x_i \Delta x_i}\left[J^{\text{agg}}_{i+1/2}({\bf{n}})- J^{\text{agg}}_{i-1/2}({\bf{n}})\right],\ \Delta {J}_{i}^{\text{brk}}({\bf{n}})=\frac{1}{x_i \Delta x_i}\left[J^{\text{brk}}_{i+1/2}({\bf{n}})- J^{\text{brk}}_{i-1/2}({\bf{n}})\right].
\end{align}
Substituting the values of $J^{\text{agg}}_{i+1/2}$ and $J^{\text{brk}}_{i+1/2}$ from equations (\ref{aggnumflux}) and (\ref{2}), respectively to get
\begin{align}\label{fluxcombagg}
\Delta x_i \Delta {J}_{i}^{\text{agg}}({\bf{n}}) &= \sum_{k=1}^{i-1} \frac{x_k}{x_i} n_k\Delta x_k \bigg(-\sum_{j=\alpha_{i-1,k}}^{\alpha_{i,k}-1} n_j \beta_{j,k} \Delta x_j + \beta_{\alpha_{i,k}-1,k} n_{\alpha_{i,k}-1} (x_{\alpha_{i,k}-1/2} - (x_{i+1/2}-x_k))\nonumber\\&- \beta_{\alpha_{i-1,k}-1,k} n_{\alpha_{i-1,k}-1}(x_{\alpha_{i-1,k}-1/2} - (x_{i-1/2}-x_k))\bigg)+  n_i\Delta x_i \Bigg( \sum_{j=\alpha_{i,i}}^{I}n_j
\beta_{j,i} \Delta x_j \nonumber\\&+
n_{\alpha_{i,i}-1} \beta_{\alpha_{i,i}-1,i} (x_{\alpha_{i,i}-1/2} - (x_{i+1/2}-x_i))\Bigg)
\end{align}
and
\begin{align}\label{fluxcombbrk}
\Delta x_i \Delta {J}_{i}^{\text{brk}}({\bf{n}}) = - \sum_{k=i+1}^I S(x_k) n_k \Delta x_k  b(x_i,x_k)\Delta x_i + S(x_i) n_i \Delta x_i \sum_{j=1}^{i-1} \frac{x_j}{x_i} b(x_j,x_i)\Delta x_j.
\end{align}
By denoting the vector ${\bf{\hat{n}}}:= [\hat{n}_1,\ldots,\hat{n}_I]$ for the numerical approximations of the average values of $n(t,x)$, the equation (\ref{E:GenDis}) can be rewritten as
\begin{align}\label{eq.FVS_SemiDiscrete1}
\frac{d {\bf{\hat{n}}}(t)}{d t}= - \left[\Delta {\bf{J}}^{\text{agg}}({\bf{\hat{n}}}) +\Delta {\bf{J}}^{\text{brk}}({\bf{\hat{n}}}) \right] = {\bf{J}}({\bf{\hat{n}}}) \pp
\end{align}

In order to retain the overall high accuracy, the semi-discrete scheme (\ref{eq.FVS_SemiDiscrete1}) can be combined
with any higher order time integration method. It is worth to mention here that dealing with the pure cases of aggregation or breakage is easy 
by setting one of the two numerical fluxes is zero. 

\section{Convergence analysis}\label{convergence:section3}
Before discussing the convergence of the semi-discrete scheme, let us review some useful definitions and theorems 
from \cite{Hundsdorfer:2003, Linz:1975} that will be used in the subsequent analysis. Let $\|\cdot\|$ denote the discrete $L^1$ norm on $\mathbb{R}^I$ that is defined as
\begin{align}\label{norm}
\|{\bf{\hat{n}}}(t)\| = \sum_{i=1}^I |\hat{n}_i(t)|\Delta x_i.
\end{align}
In this work, we deal with this norm by interpreting the discrete data as step functions.
\begin{defin}
The \textbf{spatial truncation error} is defined by the residual left by substituting the exact solution ${\bf{n}}(t)=[n_1(t),
\ldots, n_I(t)]$ into equation (\ref{eq.FVS_SemiDiscrete1}) as
\begin{align}\label{E:SpatialTruncationError}
{\boldsymbol{\sigma}}(t) = \frac{d {\bf{n}}(t)}{d t} + (\Delta {\bf{J}}^{\text{agg}}({\bf{n}})+ \Delta {\bf{J}}^{\text{brk}}({\bf{n}})) \pp
\end{align}
The scheme (\ref{eq.FVS_SemiDiscrete1}) is called consistent of order $p$ if, for $\Delta x \to 0$,
\begin{align*}
\| {\boldsymbol{\sigma}}(t) \| = {\cal O}(\Delta x^p)\pk \quad \mbox{uniformly for all } t \pk \quad 0 \leq t \leq T \pp
\end{align*}
\end{defin}

\begin{defin}\label{chura3}
The \textbf{global discretization error} is defined by ${\boldsymbol{\epsilon}}(t) = {\bf{n}}(t) - {\bf{\hat{n}}}(t)$. The scheme
(\ref{eq.FVS_SemiDiscrete1}) is called convergent of order $p$ if, for $\Delta x \to 0$,
\begin{align*}
\| {\boldsymbol{\epsilon}}(t) \| = {\cal O}(\Delta x^p)\pk \quad \mbox{uniformly for all } t \pk \quad 0 \leq t \leq T \pp
\end{align*}
\end{defin}

It is important that our numerical solution remains non-negative for all times. This is guaranteed by the next well 
known theorem where we have ${\bf{\hat{M}}}\geq 0$ for a vector ${\bf{\hat{M}}}\in \mathbb{R}^I$ iff all its components are non-negative.
\begin{thm} (Hundsdorfer and Verwer \cite[Chap. 1, Theorem 7.1]{Hundsdorfer:2003}). Suppose that $\Delta {\bf{J}}^{\text{agg}}({\bf{\hat{n}}})$ and $\Delta {\bf{J}}^{\text{brk}}({\bf{\hat{n}}})$ are continuous and satisfy the Lipschitz conditions
\begin{align*}
\|\Delta {\bf{J}}^{\text{agg}}({\bf{\hat{n}}})-\Delta {\bf{J}}^{\text{agg}}({\bf{\hat{m}}})\| \leq L_1 \|{\bf{\hat{n}}}-{\bf{\hat{m}}}\|  \ \ \ \ \text{for all}\ \ \ \ {\bf{\hat{n}}}, {\bf{\hat{m}}} \in \mathbb{R}^I
\end{align*}
and
\begin{align*}
\|\Delta {\bf{J}}^{\text{brk}}({\bf{\hat{n}}})-\Delta {\bf{J}}^{\text{brk}}({\bf{\hat{m}}})\| \leq L_2 \|{\bf{\hat{n}}}-{\bf{\hat{m}}}\|  \ \ \ \ \text{for all}\ \ \ \ {\bf{\hat{n}}}, {\bf{\hat{m}}} \in \mathbb{R}^I.
\end{align*}
Then the solution of the semi-discrete system (\ref{E:GenDis}) is non-negative if and only if for any vector ${\bf{\hat{n}}}\in \mathbb{R}^I$ and all $i= 1, \ldots, I$ and $t\geq 0$,
$${\bf{\hat{n}}}\geq 0,\quad \hat{n}_i=0 \quad \Longrightarrow \quad J_i({\bf{\hat{n}}}) \geq 0.$$
\end{thm}

Now we state a useful theorem from Linz \cite{Linz:1975} which we use to show that the FVS is convergent.
\begin{thm}\label{aaj2} Let us assume that a Lipschitz condition on ${\bf{J(n)}}$ is satisfied for $0\leq t\leq T$ and for all ${\bf{n}}, {\bf{\hat{n}}}\in \mathbb{R}^I$ where ${\bf{n}}$ and ${\bf{\hat{n}}}$ are the projected exact and numerical solutions defined in (\ref{E:FormCLtrun}) and (\ref{eq.FVS_SemiDiscrete1}), respectively. More precisely there exists a Lipschitz constant $L < \infty$ such that
\begin{align}\label{lipcond}
\|{\bf{J(n)}}- {\bf{J(\hat{n})}}\|\leq L\, \|{\bf{n}}-{\bf{\hat{n}}}\|,
\end{align}
holds. Then a consistent discretization method is also convergent and the convergence is of the same order as the consistency.
\end{thm}
\begin{proof}
A more general result is proven in Linz \cite{Linz:1975}.
\end{proof}

Due to Theorem \ref{aaj2}, for the convergence of our scheme it remains to show that the method is consistent and the Lipschitz condition (\ref{lipcond}) is satisfied by the fluxes.
\subsection{Consistency}
The following lemma gives the consistency order of the FVS for aggregation-breakage PBEs.
\begin{lem}\label{conserror}
Consider the function $S \in \mathcal{C}^2(]0, x_{\text{max}}])$ and $b, \beta \in \mathcal{C}^2(]0, x_{\text{max}}]\times ]0, x_{\text{max}}])$. Then, for any family of meshes, the consistency of the semi-discrete scheme
(\ref{eq.FVS_SemiDiscrete1}) is of second order for the pure breakage process, 
i.e.\ with $\Delta {\bf{J}}^{\text{agg}}({\bf{\hat{n}}})=0$. For the aggregation and coupled processes, 
the scheme is second order consistent on uniform and non-uniform smooth meshes while on oscillatory and random meshes 
it is first order consistent.
\end{lem}

\begin{proof} The spatial truncation error (\ref{E:SpatialTruncationError}) is given by
\begin{align}\label{abhi1}
\sigma_i(t)  = \frac{d n_i(t)}{d t} + (\Delta {J_i}^{\text{agg}}({\bf{n}})+ \Delta {J_i}^{\text{brk}}({\bf{n}})).
\end{align}
Integrating (\ref{E:FormCLtrun}) over $\Lambda_i$ and applying the mid-point rule in the time derivative term, we interpret $$\frac{d n_i(t)}{d t}= \frac{-1}{x_i \Delta x_i}\bigg[F^{\text{agg}}(x_{i+1/2}) - F^{\text{agg}}(x_{i-1/2}) +F^{\text{brk}}(x_{i+1/2}) - F^{\text{brk}}(x_{i-1/2}) \bigg]+ {\cal O}(\Delta x^2).$$ Substituting this into the equation (\ref{abhi1}) and using (\ref{naya1}) give the following form
\begin{align}\label{consisthurs}
\sigma_i(t)  = \frac{-1}{x_i \Delta x_i}\bigg[&F^{\text{agg}}(x_{i+1/2}) - F^{\text{agg}}(x_{i-1/2}) - J^{\text{agg}}_{i+1/2}({\bf{n}}) + J^{\text{agg}}_{i-1/2}({\bf{n}})\nonumber\\&+F^{\text{brk}}(x_{i+1/2})- F^{\text{brk}}(x_{i-1/2}) - J^{\text{brk}}_{i+1/2}({\bf{n}}) + J^{\text{brk}}_{i-1/2}({\bf{n}})\bigg] + {\cal O}(\Delta x^2) \nonumber\\= \sigma_i^{\text{agg}}(t)&+ \sigma_i^{\text{brk}}(t)+ {\cal O}(\Delta x^2)\pp
\end{align}
Let us now begin with
\begin{align*}
F^{\text{brk}}(x_{i+1/2}) - F^{\text{brk}}(x_{i-1/2}) = - \Bigg( & \sum_{k=i+1}^{I} \int_{\Lambda_k} S(\epsilon) n(t, \epsilon)
\int_0^{x_{i+1/2}} u b(u, \epsilon) \dd u \dd \epsilon \nonumber \\ & - \sum_{k=i}^{I} \int_{\Lambda_k}
S(\epsilon) n(t, \epsilon) \int_0^{x_{i-1/2}} u b(u, \epsilon) \dd u \dd \epsilon \Bigg).
\end{align*}
We now use Taylor series expansion of the functions $\mathcal{K}_{x_{i\pm 1/2}}(\epsilon):=n(t, \epsilon) \int_0^{x_{i\pm 1/2}} u b(u,\epsilon) \dd u$ about $x_k$ and further rearrangement of terms yield $\sigma_i^{\text{brk}}(t)$ as
\begin{align*}
\sigma_i^{\text{brk}}(t) = \frac{1}{x_i \Delta x_i} \Bigg( \sum_{k=i+1}^{I} &\left[\mathcal{K}'_{x_{i+1/2}}(x_k) - \mathcal{K}'_{x_{i-1/2}}(x_k)\right] \int_{\Lambda_k}
S(\epsilon) (\epsilon - x_k)\dd \epsilon \\ & - \mathcal{K}'_{x_{i-1/2}}(x_i) \int_{\Lambda_i}
S(\epsilon) (\epsilon - x_i)\dd \epsilon + {\cal O}(\Delta x^3) \Bigg).
\end{align*}
Applying the mid-point rule, it should be noted that
\begin{align*}
\int_{\Lambda_k} S(\epsilon) (\epsilon - x_k)\dd \epsilon = {\cal O}(\Delta x^3) \quad \text{and} \quad \mathcal{K}'_{x_{i+1/2}}(x_k) - \mathcal{K}'_{x_{i-1/2}}(x_k) = {\cal O}(\Delta x) \pp
\end{align*}
Thus we obtain $\sigma_i^{\text{brk}}(t) =  {\cal O}(\Delta x^2).$ Hence, for the pure breakage process, the consistency of the semi-discrete scheme (\ref{eq.FVS_SemiDiscrete1}) is two which is determined by using (\ref{norm}) as
\begin{align*}
\| {\boldsymbol{\sigma}}(t) \| =  \sum_{i=1}^{I} |\sigma_i^{\text{brk}}(t)|\Delta x_i =  {\cal O}(\Delta x^2) \pk
\end{align*}
independently of the type of meshes.\enter

Due to the non-linearity of the aggregation problem, it is not easy to determine the consistency order on general 
meshes and therefore, we evaluate it on various meshes separately. The results can be combined to the results of breakage process to give the consistency of the coupled processes. We know from (\ref{aggfluxmon26})
\begin{align*}
F^{\text{agg}}(x_{i+1/2}) - F^{\text{agg}}(x_{i-1/2}) = - \Bigg( & \sum_{j=1}^{i} \int_{\Lambda_j} u\, n(t,u) \int_{x_{i+1/2}-u}^{x_{\text{max}}} \beta(u,v)n(t,v)dv du\\& - \sum_{j=1}^{i-1} \int_{\Lambda_j} u\, n(t,u) \int_{x_{i-1/2}-u}^{x_{\text{max}}} \beta(u,v)n(t,v)dv du \Bigg).
\end{align*}
Define $\mathcal{L}_{x_{i\pm 1/2}}(u):= n(t,u) \int_{x_{i\pm 1/2}-u}^{x_{\text{max}}}  \beta(u,v)n(t,v)dv$. Taylor series expansion of the functions $\mathcal{L}_{x_{i\pm 1/2}}(u)$ about $x_j$ gives
\begin{align}\label{tue2707}
F^{\text{agg}}(x_{i+1/2}) - F^{\text{agg}}(x_{i-1/2}) =& \Bigg( \sum_{j=1}^{i} \int_{\Lambda_j} u \left(\mathcal{L}_{x_{i+ 1/2}}(x_j)+ (u-x_j)\mathcal{L}_{x_{i+ 1/2}}^{'}(x_j)\right) du\nonumber\\ - \sum_{j=1}^{i-1} \int_{\Lambda_j} & u \left(\mathcal{L}_{x_{i- 1/2}}(x_j)+ (u-x_j)\mathcal{L}_{x_{i- 1/2}}^{'}(x_j)\right) du \Bigg)+ {\cal O}(\Delta x^3).
\end{align}
Applying the mid-point rule, it should again be noted that
\begin{align*}
\int_{\Lambda_j} u (u -x_j) \dd u = {\cal O}(\Delta x^3) \quad \text{and} \quad \mathcal{L}'_{x_{i+1/2}}(x_j) - \mathcal{L}'_{x_{i-1/2}}(x_j) = {\cal O}(\Delta x) \pp
\end{align*}
Therefore, by defining $LHS:= F^{\text{agg}}(x_{i+1/2}) - F^{\text{agg}}(x_{i-1/2})$, the equation (\ref{tue2707}) reduces to
\begin{align*}
LHS= \Bigg( \sum_{j=1}^{i} \int_{\Lambda_j} u \mathcal{L}_{x_{i+ 1/2}}(x_j) du- \sum_{j=1}^{i-1} \int_{\Lambda_j} u \mathcal{L}_{x_{i- 1/2}}(x_j) du \Bigg)+ {\cal O}(\Delta x^3).
\end{align*} 
Substituting the values of $\mathcal{L}_{x_{i\pm 1/2}}(x_j)$ yield (leaving the third order terms) 
\begin{align*}
LHS= \Bigg( \underbrace{\sum_{j=1}^{i} \int_{\Lambda_j} u n_j \int_{x_{i+ 1/2}-x_j}^{x_{\text{max}}}  \beta(x_j,v)n(t,v)dv du}_{I_1}- \underbrace{\sum_{j=1}^{i-1} \int_{\Lambda_j} u n_j \int_{x_{i-1/2}-x_j}^{x_{\text{max}}} \beta(x_j,v)n(t,v)dv du}_{I_2} \Bigg).
\end{align*}
Now, $I_1$ is equivalent to
\begin{align*}
I_1=& \sum_{j=1}^{i} \int_{\Lambda_j} u n_j \left[\int_{x_{i+ 1/2}-x_j}^{x_{\alpha_{i,j}-1/2}} + \sum_{k=\alpha_{i,j}}^I \int_{\Lambda_k} \right] \beta(x_j,v)n(t,v)dv du.
\end{align*}
Applying the mid-point approximation for the second term, we figure out
\begin{align*}
I_1=& \sum_{j=1}^{i} x_j n_j \Delta x_j \bigg[\int_{x_{i+ 1/2}-x_j}^{x_{\alpha_{i,j}-1/2}}\beta(x_j,v)n(t,v)dv \\& + \sum_{k=\alpha_{i,j}}^I \beta_{j,k} n_k \Delta x_k+  \sum_{k=\alpha_{i,j}}^I \int_{\Lambda_k} (v-x_k)^2/2 (\beta(x_j,v)n(t,v))^{''}\bigg]dv+ {\cal O}(\Delta x^3).
\end{align*}
Similarly, we estimate
\begin{align*}
I_2=& \sum_{j=1}^{i-1} x_j n_j \Delta x_j \bigg[\int_{x_{i-1/2}-x_j}^{x_{\alpha_{i-1,j}-1/2}}\beta(x_j,v)n(t,v)dv\\& + \sum_{k=\alpha_{i-1,j}}^I \beta_{j,k} n_k \Delta x_k+  \sum_{k=\alpha_{i-1,j}}^I \int_{\Lambda_k} (v-x_k)^2/2 (\beta(x_j,v)n(t,v))^{''}\bigg]dv+ {\cal O}(\Delta x^3).
\end{align*}
Subtracting the third term from $I_2$ to $I_1$ gives
\begin{align*}
\bigg[\sum_{j=1}^{i} \sum_{k=\alpha_{i,j}}^I - \sum_{j=1}^{i-1} \sum_{k=\alpha_{i-1,j}}^I \bigg] x_j n_j \Delta x_j \int_{\Lambda_k} (v-x_k)^2/2 (\beta(x_j,v)n(t,v))^{''}dv= \\ \bigg[-\sum_{j=1}^{i-1} \sum_{k=\alpha_{i-1,j}}^{k=\alpha_{i,j}-1}\bigg] x_j n_j \Delta x_j \int_{\Lambda_k} (v-x_k)^2/2 (\beta(x_j,v)n(t,v))^{''}dv+ {\cal O}(\Delta x^3).
\end{align*}
By using Lemma \ref{aaj1} which is stated in the next section, the summation over $k$ is finite in this term. Hence, 
the rhs of this equation becomes of order ${\cal O}(\Delta x^3)$ and can be omitted. Therefore,
\begin{align*}
LHS= \sum_{j=1}^{i} x_j n_j \Delta x_j \bigg[\underbrace{\int_{x_{i+ 1/2}-x_j}^{x_{\alpha_{i,j}-1/2}}\beta(x_j,v)n(t,v)dv}_{I_3} + \sum_{k=\alpha_{i,j}}^I \beta_{j,k} n_k \Delta x_k\bigg]\\- \sum_{j=1}^{i-1} x_j n_j \Delta x_j   \bigg[\underbrace{\int_{x_{i-1/2}-x_j}^{x_{\alpha_{i-1,j}-1/2}}\beta(x_j,v)n(t,v)dv}_{I_4} + \sum_{k=\alpha_{i-1,j}}^I \beta_{j,k} n_k \Delta x_k\bigg] \Bigg)+{\cal O}(\Delta x^3).
\end{align*}
Open the Taylor series about the points $x_{\alpha_{i,j}-1}$ in $I_3$ and $x_{\alpha_{i-1,j}-1}$ in $I_4$ as well as by using the relation (\ref{aggnumflux}), we finally obtain
\begin{align*}
LHS=\Bigg( J^{\text{agg}}_{i+1/2}+\sum_{j=1}^{i} x_j n_j \Delta x_j \int_{x_{i+ 1/2}-x_j}^{x_{\alpha_{i,j}-1/2}} (v-x_{\alpha_{i,j}-1})(\beta(x_j,v)n(t,v))^{'}\lvert_{v=x_{\alpha_{i,j}-1}} dv\\ -J^{\text{agg}}_{i-1/2}- \sum_{j=1}^{i-1}x_j n_j \Delta x_j\int_{x_{i- 1/2}-x_j}^{x_{\alpha_{i-1,j}-1/2}} (v-x_{\alpha_{i-1,j}-1})(\beta(x_j,v)n(t,v))^{'}\lvert_{v=x_{\alpha_{i-1,j}-1}} dv\Bigg)+{\cal O}(\Delta x^3).
\end{align*}
Let $f(x_j,v)= \beta(x_j,v)n(t,v)$ and $\frac{\partial f}{\partial v}\lvert_{v = x_{\alpha_{i,j}}} = f'(x_j,x_{\alpha_{i,j}})$. This implies that
\begin{align}\label{finalcalculation}
\sigma_i^{\text{agg}}(t) = \frac{1}{x_i \Delta x_i}& \bigg[\sum_{j=1}^{i} x_j n_j \Delta x_j \int_{x_{i+ 1/2}-x_j}^{x_{\alpha_{i,j}-1/2}} (v-x_{\alpha_{i,j}-1}) f'(x_j,x_{\alpha_{i,j}-1}) dv\nonumber\\ - \sum_{j=1}^{i-1} & x_j n_j \Delta x_j\int_{x_{i- 1/2}-x_j}^{x_{\alpha_{i-1,j}-1/2}} (v-x_{\alpha_{i-1,j}-1})f'(x_j,x_{\alpha_{i-1,j}-1}) dv \bigg] + {\cal O}(\Delta x^2).
\end{align}
Now the consistency order on four different types of meshes are evaluated:
\subsubsection{Uniform mesh}
Let us assume that the first mesh is uniform, i.e.\ $\Delta x_i= \Delta x$ for all $i$. In this case $x_{i+ 1/2}-x_j$ and $x_{\alpha_{i,j}-1}$ become the same and are equal to the pivot point $x_{i-j+1}$. Similarly, 
\begin{align}\label{griduniformreln}
x_{i- 1/2}-x_j=x_{\alpha_{i-1,j}-1} =x_{i-j}. 
\end{align}
Applying the Taylor series expansion of the function $f'(x_j,x_{\alpha_{i-1,j}-1}+(x_{\alpha_{i,j}-1}- x_{\alpha_{i-1,j}-1}))$ about the point $x_{\alpha_{i-1,j}-1}$ in the first term on the rhs of the equation (\ref{finalcalculation}) to get
\begin{align*}
\sigma_i^{\text{agg}}(t) = \frac{1}{x_i \Delta x_i}& \bigg[\sum_{j=1}^{i-1} x_j n_j \Delta x_j f'(x_j,x_{\alpha_{i-1,j}-1}) \bigg(\int_{x_{i+ 1/2}-x_j}^{x_{\alpha_{i,j}-1/2}} (v-x_{\alpha_{i,j}-1}) dv\\& - \int_{x_{i- 1/2}-x_j}^{x_{\alpha_{i-1,j}-1/2}} (v-x_{\alpha_{i-1,j}-1}) dv\bigg) \bigg] + {\cal O}(\Delta x^2).
\end{align*}
Further by facilitating the integrals and using the relation (\ref{griduniformreln}), we have
\begin{align*}
\sigma_i^{\text{agg}}(t) = \frac{1}{x_i \Delta x_i} \bigg[\sum_{j=1}^{i-1} x_j n_j \Delta x_j f'(x_j,x_{\alpha_{i-1,j}-1}) \bigg(\frac{\Delta x_{\alpha_{i,j}-1}^2}{8} - \frac{\Delta x_{\alpha_{i-1,j}-1}^2}{8}\bigg) \bigg] + {\cal O}(\Delta x^2).
\end{align*}
Hence, $\sigma_i^{\text{agg}}(t) =  {\cal O}(\Delta x^2)$ and so the order of consistency is given by using (\ref{norm}) as
\begin{align*}
\| {\boldsymbol{\sigma}}(t) \| =  \sum_{i=1}^{I} |\sigma_i^{\text{agg}}(t)|\Delta x_i =  {\cal O}(\Delta x^2) \pp
\end{align*}
Therefore, the scheme is second order consistent on uniform grids.
\subsubsection{Non-uniform smooth mesh}
A smooth transformation from uniform grids leads to such meshes. In this case grids are assumed to be smooth in the sense that $\Delta x_i - \Delta x_{i-1} = {\cal O}(\Delta x^2)$ and $2 \Delta x_i - (\Delta x_{i-1}+\Delta x_{i+1}) = {\cal O}(\Delta x^3)$, where $\Delta x$ is the maximum mesh width. For example, let us consider a variable $\xi$ with uniform mesh and a smooth transformation $x=g(\xi)$ to get non-uniform smooth mesh, see Figure \ref{p:SmoothMesh}. For the analysis here, we have considered the exponential transformation as $x= \exp(\xi).$ Such a mesh is also known as a geometric mesh, i.e.\ $x_{i+1/2}= r x_{i-1/2}$ with $r= \exp(\bar{h})$. The term $\bar{h}$ is the width of the uniform grid. Here again we achieve second order consistency. \enter
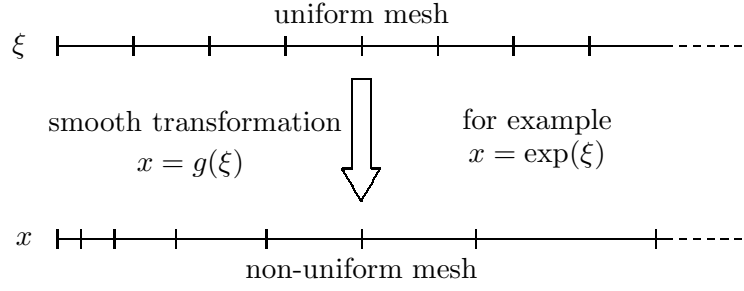
\begin{figure}
\centering
\input{UniformMesh}
\caption{Non-uniform smooth mesh.}\label{p:SmoothMesh}
\end{figure}

Equation (\ref{finalcalculation}) can be rewritten by setting $j=j-1$ in second term as 
\begin{align*}
\sigma_i^{\text{agg}}(t) =& \frac{1}{x_i \Delta x_i} \bigg[\underbrace{\sum_{j=1}^{i} x_j n_j \Delta x_j \int_{x_{i+ 1/2}-x_j}^{x_{\alpha_{i,j}-1/2}} (v-x_{\alpha_{i,j}-1}) f'(x_j,x_{\alpha_{i,j}-1}) dv}_{A}\\ -&\underbrace{\sum_{j=2}^{i} x_{j-1} n_{j-1} \Delta x_{j-1} \int_{x_{i- 1/2}-x_{j-1}}^{x_{\alpha_{i-1,j-1}-1/2}} (v-x_{\alpha_{i-1,j-1}-1})f'(x_{j-1},x_{\alpha_{i-1,j-1}-1}) dv}_{B} \bigg] + {\cal O}(\Delta x^2).
\end{align*}
Now we simplify $A-B$ as
\begin{align*}
A-B = & \sum_{j=2}^{i} x_{j-1} n_{j-1} \Delta x_j \int_{x_{i+1/2}-x_j}^{x_{\alpha_{i,j}-1/2}} (v-x_{\alpha_{i,j}-1})f'(x_{j-1},x_{\alpha_{i-1,j-1}-1}) dv\\
&-\sum_{j=2}^{i} x_{j-1} n_{j-1} \Delta x_{j-1} \int_{x_{i-1/2}-x_{j-1}}^{x_{\alpha_{i-1,j-1}-1/2}} (v-x_{\alpha_{i-1,j-1}-1}) f'(x_{j-1},x_{\alpha_{i-1,j-1}-1}) dv + {\cal O}(\Delta x^3).
\end{align*}
Further it can be rewritten as
\begin{align*}
A-B = & \sum_{j=2}^{i} x_{j-1} n_{j-1} (\Delta x_j-\Delta x_{j-1}) \int_{x_{i+1/2}-x_j}^{x_{\alpha_{i,j}-1/2}} (v-x_{\alpha_{i,j}-1})f'(x_{j-1},x_{\alpha_{i-1,j-1}-1}) dv\\
&+\sum_{j=2}^{i} x_{j-1} n_{j-1} \Delta x_{j-1} \int_{x_{i+1/2}-x_j}^{x_{\alpha_{i,j}-1/2}} (v-x_{\alpha_{i,j}-1})f'(x_{j-1},x_{\alpha_{i-1,j-1}-1}) dv\\
&-\sum_{j=2}^{i} x_{j-1} n_{j-1} \Delta x_{j-1} \int_{x_{i-1/2}-x_{j-1}}^{x_{\alpha_{i-1,j-1}-1/2}} (v-x_{\alpha_{i-1,j-1}-1}) f'(x_{j-1},x_{\alpha_{i-1,j-1}-1}) dv + {\cal O}(\Delta x^3).
\end{align*}
For such smooth meshes, $\Delta x_j-\Delta x_{j-1} = {\cal O}(\Delta x^2)$ holds. Setting 
${\alpha_{i,j}-1} = \alpha_1$, ${\alpha_{i-1,j-1}-1} = \alpha_2$ and $g_{i,j} =x_{j-1} n_{j-1} 
\Delta x_{j-1} f'(x_{j-1},x_{\alpha_{i-1,j-1}-1})$ yield
\begin{align*}
A-B =\sum_{j=2}^{i} g_{i,j} \left(\int_{x_{i+1/2}-x_j}^{x_{\alpha_{1}+1/2}} (v-x_{\alpha_{1}})dv-\int_{x_{i-1/2}-x_{j-1}}^{x_{\alpha_{2}+1/2}} (v-x_{\alpha_{2}})dv\right) + {\cal O}(\Delta x^3).
\end{align*}
It can further be simplified as
\begin{align*}
A-B =\sum_{j=2}^{i}\frac{g_{i,j}}{2} \left( \frac{\Delta x_{\alpha_1}^2}{4} - \frac{\Delta x_{\alpha_2}^2}{4}
+ \left[({x_{i-1/2}-x_{j-1}})-x_{\alpha_{2}}\right]^2- \left[({x_{i+1/2}-x_j})-x_{\alpha_{1}}\right]^2\right)+ {\cal O}(\Delta x^3).
\end{align*}
Since $x_{i+1/2}-x_j \in \Lambda_{\alpha_{i,j}-1}$, thus $x_{i-1/2}-x_{j-1} \in \Lambda_{\alpha_{i-1,j-1}-1}$. Further notice that
$x_{i+1/2}-x_j = r (x_{i-1/2}-x_{j-1})$ and therefore $\alpha_1 = \alpha_2+1$. Again by using the 
condition $\Delta x_j-\Delta x_{j-1} = {\cal O}(\Delta x^2)$, we determine 
$\Delta x_{\alpha_1}^2 - \Delta x_{\alpha_2}^2 = {\cal O}(\Delta x^3)$. Now, to get a second order 
consistency of the scheme, it is remained to show that $$ \left[({x_{i-1/2}-x_{j-1}})-x_{\alpha_{2}}\right]^2- \left[({x_{i+1/2}-x_j})-x_{\alpha_{1}}\right]^2 = {\cal O}(\Delta x^3)$$ or equivalently,
\begin{align}\label{thurs2907}
\left[({x_{i-1/2}-x_{j-1}})-x_{\alpha_{2}}\right]- \left[({x_{i+1/2}-x_j})-x_{\alpha_{1}}\right] = {\cal O}(\Delta x^2).
\end{align}
Let us consider $\xi_1$, $\xi_2$ are corresponding points in the uniform mesh for $x_{\alpha_2}$ and $x_{i-1/2}-x_{j-1}$, respectively.
Consider $h_1= \xi_2-\xi_1$ which is given as $$h_1= \xi_2-\xi_1= \log\left(x_{i-1/2}-x_{j-1}\right)- \log\left(x_{\alpha_2}\right)= \log\left(\frac{x_{i-1/2}-x_{j-1}}{x_{\alpha_2}}\right).$$
Similarly, taking $h_2= \xi_4-\xi_3$ where $\xi_3$ and $\xi_4$ are the points in the uniform mesh 
corresponding to the points $x_{\alpha_1}$ and $x_{i+1/2}-x_j$, respectively, 
we evaluate 
$$h_2= \xi_4-\xi_3= \log\left(x_{i+1/2}-x_{j}\right)- \log\left(x_{\alpha_1}\right)= \log\left(\frac{x_{i+1/2}-x_j}{x_{\alpha_1}}\right)= \log\left(\frac{x_{i-1/2}-x_{j-1}}{x_{\alpha_2}}\right)= h_1.$$

Setting $h= h_1= h_2.$ Further 
$$\xi_3-\xi_1= \log\left(x_{\alpha_{1}}\right)-\log\left(x_{\alpha_{2}}\right)= \log\left(\frac{x_{\alpha_{1}}}{x_{\alpha_{2}}}\right)= \log\left(r\right)= \bar{h}.$$
Finally, the equation (\ref{thurs2907}) can be estimated by using Taylor series expansion as 
\begin{align*}
\left[({x_{i-1/2}-x_{j-1}})-x_{\alpha_{2}}\right]- \left[({x_{i+1/2}-x_j})-x_{\alpha_{1}}\right]=& \left[g(\xi_2)-g(\xi_1)\right]- \left[g(\xi_4)-g(\xi_3)\right]\\
=& h g^{'}(\xi_1) - h g^{'}(\xi_3) + {\cal O}(h^2)\\
=& h(g^{'}(\xi_1)- g^{'}(\xi_1+\bar{h})) + {\cal O}(h^2)\\ =& -h\bar{h}g^{''}(\xi_1)+{\cal O}(h^2)= {\cal O}(h^2). 
\end{align*}
Hence, by using (\ref{consisthurs}) and (\ref{norm}) the order of consistency for the pure aggregation process is two for the smooth meshes $x_{i+1/2}=r x_{i-1/2}$.
\subsubsection{Oscillatory and random meshes}
A mesh is known to be an oscillatory mesh, if for $r > 0 (r\neq 1)$ it is given as
\begin{align}\label{oscimesh}
\Delta x_{i+1}:= \begin{cases} {r \Delta x_{i}} \ \ \ \ \text{if}\ \ \ i \ \ \text{is odd},\\
                 {\frac{1}{r} \Delta x_{i}} \ \ \ \ \text{if}\ \ \ i \ \ \text{is even}.
                \end{cases}
\end{align}
From the equation (\ref{finalcalculation}), it is clear that the first two terms on the rhs can not be cancel out for an oscillatory or a random mesh. Therefore, $\sigma_i^{\text{agg}}(t) = {\cal O}(\Delta x)$ and so the accuracy of the semi discrete scheme (\ref{eq.FVS_SemiDiscrete1}) is one by using the relation (\ref{norm}) on such meshes.\enter

Now for the coupled aggregation and breakage problems, the local truncation error of each process can be combined and give second order consistency on uniform and non-uniform smooth meshes whereas it is of first order on the other two types of grids.
\end{proof}

\subsection{Lipschitz continuity of the fluxes}
To prove the Lipschitz continuity of the numerical flux ${\bf{J(\hat{n})}}$ in 
(\ref{eq.FVS_SemiDiscrete1}), the following three lemmas are used.
\begin{lem}\label{aaj1} Let us assume that the points $x_{i+j-\frac{1}{2}}- x_k$ for given $i, k$ and $j= 1,2,\ldots, p$ where $p\geq 2$ lie in the same cell $\Lambda_\alpha$ for some index $\alpha$. We also assume that our grid satisfies the quasi-uniformity condition
\begin{align}\label{quasiuniformgrid}
\frac{\Delta x_{\text{max}}} {\Delta x_{\text{min}}} \leq C
\end{align}
for some constant $C$(independent of the mesh size). Then $p$ is bounded by $C+1$.
\end{lem}
\begin{proof}
Our assumption on the points implies that by (\ref{indexagg}), we have
\begin{align*}
{\alpha_{i,k}-1} = {\alpha_{i+1,k}-1} = \ldots = {\alpha_{i+p-1,k}-1} = \alpha.
\end{align*}
Clearly, $\Delta x_\alpha \geq \Delta x_{i+1} + \Delta x_{i+2}+\ldots +\Delta x_{i+p-1}.$
This implies that $$\frac{\Delta x_{\alpha}} {\Delta x_l} \leq \frac{\Delta x_{\text{max}}} {\Delta x_{\text{min}}} \leq C \quad \Rightarrow \quad \frac{\Delta x_{\alpha}} {C} \leq \Delta x_l \quad \text{for}\quad l= i+1,\cdots,i+p-1.$$
Therefore, $\Delta x_{\alpha} \geq (p-1) \frac{\Delta x_{\alpha}}{C}$, giving $p \leq (C+1).$
\end{proof}

In the next two lemmas the boundedness of the total number of particles for the aggregation and 
multiple breakage equations are discussed.

\begin{lem}\label{boundednessofnumbercontinuous} Let us assume that the kernels $\beta$, $S$ and $b$ 
satisfy the boundedness condition (\ref{kernelbound}). Then the total number of particles for the continuous aggregation-breakage equation (\ref{E:FormCLtrun}) is
bounded by a constant $C_{T,x_{\text{max}}}> 0$ depending on $T$ and $x_{\text{max}}$, namely
\begin{align*}
\int_0^{x_{\text{max}}} n(t,x)dx= N(t) =\sum_{i=1}^I N_i(t) \leq {N}(0) \exp(x_{\text{max}} {Q}_1 T) = C_{T,x_{\text{max}}}.
\end{align*}
\end{lem}
\begin{proof}
The proof can be found in Appendix \ref{appendixFV_number}.
\end{proof}

\begin{lem}\label{boundednessofnumber} Under the same assumptions on $\beta$, $S$ and $b$ considered in the previous lemma, we have boundedness of the total number of particles for the discrete aggregation-breakage equation (\ref{E:GenDis}) by using the finite volume scheme. The bound in this case is again $C_{T,x_{\text{max}}}$ as before, i.e.\ 
\begin{align}\label{numboundonnumber}
\sum_{i=1}^I \hat{n}_i \Delta x_i = \hat{N}(t)= \sum_{i=1}^I \hat{N}_i(t) \leq \hat{N}(0) \exp(x_{\text{max}} {Q}_1 T) = C_{T,x_{\text{max}}}
\end{align}
provided that the initial data $\hat{N}(0)$ and ${N}(0)$ are the same.
\end{lem}
\begin{proof}
The proof has been given in Appendix \ref{appendixFV_number}.
\end{proof}

Now, the Lipschitz continuity of the numerical flux ${\bf{J(\hat{n})}}$ defined as in 
(\ref{eq.FVS_SemiDiscrete1}) is shown.
\begin{lem}\label{consislemma25081}
Let us assume that our grid satisfies the quasi-uniformity condition (\ref{quasiuniformgrid}). We also assume that the kernels 
$\beta$, $S$ and $b$ satisfy the bounds (\ref{kernelbound}) which are $\beta\leq Q$ and $bS\leq Q_1$. Then there exists a Lipschitz constant $L:= (4C+6)Q C_{T,x_{\text{max}}} + 2 Q_1 x_{\text{max}}< \infty$ for some constants $C, C_{T,x_{\text{max}}} > 0$ such that
\begin{align}\label{lipcondnew}
\|{\bf{J(n)}}- {\bf{J(\hat{n})}}\|\leq L\, \|{\bf{n}}-{\bf{\hat{n}}}\|,
\end{align}
holds.
\end{lem}
\begin{proof}
From (\ref{eq.FVS_SemiDiscrete1}), we have the following discretized form of the equation
\begin{align}\label{E:GenDis2}
\frac{d {\bf{\hat{n}}}(t)}{d t}= - \left[\Delta {\bf{J}}^{\text{agg}}({\bf{\hat{n}}}) +\Delta {\bf{J}}^{\text{brk}}({\bf{\hat{n}}}) \right] = {\bf{J}}({\bf{\hat{n}}}) \pp
\end{align}
To prove the Lipschitz conditions on ${\bf{J(\hat{n})}}$, it is sufficient to find the Lipschitz conditions 
on $\Delta {\bf{J}}^{\text{agg}}({\bf{\hat{n}}})$ and $\Delta {\bf{J}}^{\text{brk}}({\bf{\hat{n}}})$ separately. 
For the aggregation,
\begin{align*}
\|\Delta {\bf{J}}^{\text{agg}}({\bf{{n}}})- \Delta {\bf{J}}^{\text{agg}}({\bf{\hat{n}}})\|&= \sum_{i= 1}^I \Delta x_i |\Delta J^{\text{agg}}_i({\bf{n}})- \Delta J^{\text{agg}}_i({\bf{\hat{n}}})|.
\end{align*}
Substituting the value of $\Delta J^{\text{agg}}_i({\bf{n}})$ from the equation (\ref{fluxcombagg}) yields
\begin{align}\label{21}
\|\Delta {\bf{J}}^{\text{agg}}({\bf{{n}}})- \Delta {\bf{J}}^{\text{agg}}({\bf{\hat{n}}})\| \leq \sum_{i= 1}^I \bigg|\sum_{k=1}^{i-1} \frac{x_k}{x_i} & \Delta x_k \sum_{j=\alpha_{i-1,k}}^{\alpha_{i,k}-1} \beta_{j,k} \Delta x_j (-n_j n_k+\hat{n}_j \hat{n}_k)\bigg|\nonumber\\+ \sum_{i= 1}^I \bigg|\sum_{k=1}^{i-1} \frac{x_k}{x_i} \beta_{\alpha_{i,k}-1,k}\Delta x_k & (x_{\alpha_{i,k}-1/2} - (x_{i+1/2}-x_k))(n_k n_{\alpha_{i,k}-1}-\hat{n}_k \hat{n}_{\alpha_{i,k}-1})\bigg| \nonumber\\+ \sum_{i= 1}^I \bigg|\sum_{k=1}^{i-1} \frac{x_k}{x_i} \beta_{\alpha_{i-1,k}-1,k} \Delta x_k & (x_{\alpha_{i-1,k}-1/2} - (x_{i-1/2}-x_k))(n_k n_{\alpha_{i-1,k}-1}-\hat{n}_k \hat{n}_{\alpha_{i-1,k}-1})\bigg|\nonumber\\ +\sum_{i= 1}^I\bigg(\bigg|\sum_{j=\alpha_{i,i}}^{I} \beta_{j,i} \Delta x_i \Delta x_j &(n_i n_j- \hat{n}_i \hat{n}_j) \nonumber\\+ \beta_{\alpha_{i,i}-1,i} \Delta x_i & (x_{\alpha_{i,i}-1/2} - (x_{i+1/2}-x_i))(n_i n_{\alpha_{i,i}-1}- \hat{n}_i \hat{n}_{\alpha_{i,i}-1}) \bigg|\bigg)\nonumber\\ \leq S_1 + S_2\ +\ & S_3 + S_4 .
\end{align}
Now the terms $S_i, i= 1,\cdots,4$ in (\ref{21}) are evaluated one by one. First the term $S_1$ is simplified which may be estimated
\begin{align*}
S_1 \leq& \sum_{i= 1}^I \sum_{k=1}^{i-1} \frac{x_k}{x_i} \Delta x_k \sum_{j=\alpha_{i-1,k}}^{\alpha_{i,k}-1} \beta_{j,k} \Delta x_j |n_j n_k-\hat{n}_j \hat{n}_k|.
\end{align*}

Since $k < i$ implies that $x_k < x_i$. Using the relation 
$xy-\hat{x}\hat{y} = 1/2 [(x-\hat{x})(y+\hat{y})+(x+\hat{x})(y-\hat{y})]$, bound 
$\beta(x,y)\leq Q$ and setting $N_i= n_i\Delta x_i$ give
% \begin{align*}
% S_1 \leq& \sum_{i= 1}^I \sum_{k=1}^{i-1} \frac{\Delta x_k}{2} \sum_{j=\alpha_{i-1,k}}^{\alpha_{i,k}-1} \beta_{j,k} \Delta x_j |(n_j+\hat{n}_j)(n_k-\hat{n}_k)+(n_j-\hat{n}_j)(n_k+\hat{n}_k)|.
% \end{align*}
%By using the bound $\beta(x,y)\leq Q$ and setting $N_i= n_i\Delta x_i$ give
\begin{align*}
S_1 \leq& \frac{Q}{2} \sum_{i= 1}^I \bigg(\sum_{k=1}^{I} \Delta x_k |n_k- \hat{n}_k| \sum_{j=\alpha_{i-1,k}}^{\alpha_{i,k}-1} (N_j+\hat{N}_j)+ \sum_{k=1}^{I} (N_k+ \hat{N}_k) \sum_{j=\alpha_{i-1,k}}^{\alpha_{i,k}-1} \Delta x_j |n_j-\hat{n}_j| \bigg).
\end{align*}
Open the summation for each $i$, we obtain
\begin{align*}
S_1 \leq& \frac{Q}{2} \sum_{k= 1}^I \Delta x_k |n_k- \hat{n}_k| \sum_{j=\alpha_{0,k}}^{\alpha_{I,k}-1} (N_j+\hat{N}_j)+ \frac{Q}{2} \sum_{k=1}^{I} (N_k+ \hat{N}_k) \sum_{j=\alpha_{0,k}}^{\alpha_{I,k}-1} \Delta x_j |n_j-\hat{n}_j|.
\end{align*}
Having Lemmas \ref{boundednessofnumbercontinuous} and \ref{boundednessofnumber}, which say that the total number of 
particles is bounded by a constant $C_{T,x_{\text{max}}}$, $S_1$ is further simplified as $ S_1 \leq 2 Q C_{T,x_{\text{max}}} \|n-\hat{n}\|$.\enter

Now the term $S_2$ is calculated from (\ref{21}) which is taken as
\begin{align*}
S_2 \leq \sum_{i= 1}^I \sum_{k=1}^{i-1} \frac{x_k}{x_i}\beta_{\alpha_{i,k}-1,k} \Delta x_k (x_{\alpha_{i,k}-1/2} - (x_{i+1/2}-x_k))\, |n_k n_{\alpha_{i,k}-1}-\hat{n}_k \hat{n}_{\alpha_{i,k}-1}|.
\end{align*}
Further simplifications as in the previous case yield
\begin{align*}
S_2\leq &\sum_{i= 1}^I \sum_{k=1}^{i-1} \frac{Q}{2} \Delta x_k \Delta x_{\alpha_{i,k}-1}\bigg(|(n_k - \hat{n}_k)(n_{\alpha_{i,k}-1}+\hat{n}_{\alpha_{i,k}-1})+ (n_k + \hat{n}_k)(n_{\alpha_{i,k}-1}-\hat{n}_{\alpha_{i,k}-1})|\bigg)\\ \leq  \frac{Q}{2}& \sum_{i= 1}^I \sum_{k=1}^{i-1} \Delta x_k |n_k - \hat{n}_k|\, (N_{\alpha_{i,k}-1}+\hat{N}_{\alpha_{i,k}-1}) + \frac{Q}{2} \sum_{i= 1}^I \sum_{k=1}^{i-1} \Delta x_{\alpha_{i,k}-1} |n_{\alpha_{i,k}-1}-\hat{n}_{\alpha_{i,k}-1}| (N_k+\hat{N}_k).
\end{align*}
Changing the order of summation gives
\begin{align*}
S_2\leq \frac{Q}{2} \sum_{k= 1}^I & \Delta x_k |n_k - \hat{n}_k| \sum_{i=k+1}^{I} (N_{\alpha_{i,k}-1}+\hat{N}_{\alpha_{i,k}-1})\\& + \frac{Q}{2} \sum_{k= 1}^I (N_k+\hat{N}_k) \sum_{i=k+1}^{I} \Delta x_{\alpha_{i,k}-1} |n_{\alpha_{i,k}-1}-\hat{n}_{\alpha_{i,k}-1}|.
\end{align*}
By using the Lemma \ref{aaj1} which shows that the number of repetition of index in a cell is finite and bounded by some constant $C$, 
we obtain $ S_2\leq 2 C Q C_{T,x_{\text{max}}} \|n - \hat{n}\|$. The same bound on $S_3$ is achieved because the only 
difference is that the index $i-1$ is used instead of $i$.\enter

Finally the expression $S_4$ from (\ref{21}) can be written as
\begin{align*}
S_4\leq \sum_{i=1}^I & \bigg(\sum_{j=\alpha_{i,i}}^{I} \beta_{j,i} \Delta x_i \Delta x_j |n_i n_j- \hat{n}_i \hat{n}_j|\\&+ \beta_{\alpha_{i,i}-1,i} \Delta x_i (x_{\alpha_{i,i}-1/2} - (x_{i+1/2}-x_i))\, |n_i n_{\alpha_{i,i}-1}- \hat{n}_i \hat{n}_{\alpha_{i,i}-1}|\bigg)\\ \leq \frac{Q}{2} & \sum_{i=1}^I \sum_{j=1}^{I} (N_i+\hat{N}_i) \Delta x_j |n_j-\hat{n}_j|+ \frac{Q}{2} \sum_{i=1}^I \sum_{j=1}^{I} (N_j+\hat{N}_j) \Delta x_i |n_i-\hat{n}_i|\\& + \frac{Q}{2} \sum_{i=1}^I \Delta x_i |n_i-\hat{n}_i| (N_{\alpha_{i,i}-1}+\hat{N}_{\alpha_{i,i}-1}) + \frac{Q}{2} \sum_{i=1}^I (N_i+\hat{N}_i) \Delta x_{\alpha_{i,i}-1} |n_{\alpha_{i,i}-1}-\hat{n}_{\alpha_{i,i}-1}|.
\end{align*}
Further simplification gives $ S_4\leq 4 Q C_{T,x_{\text{max}}} \|n- \hat{n}\|.$ Adding all the results from 
$S_1, S_2, S_3$ and $S_4$ yields
\begin{align}
\|\Delta {\bf{J}}^{\text{agg}}({\bf{n}})-\Delta {\bf{J}}^{\text{agg}}({\bf{\hat{n}}})\|\leq& (4 C + 6 )Q C_{T,x_{\text{max}}} \|{\bf{n}}-{\bf{\hat{n}}}\|,
\end{align}
with a Lipschitz constant $L_1= (4 C + 6 )Q C_{T,x_{\text{max}}}$.\enter

Similarly, for the breakage problem, we have
\begin{align*}
\|\Delta {\bf{J}}^{\text{brk}}({\bf{n}})-\Delta {\bf{J}}^{\text{brk}}({\bf{\hat{n}}})\|&= \sum_{i=1}^I \Delta x_i \left|\Delta J^{\text{brk}}_{i}({\bf{n}})-\Delta J^{\text{brk}}_{i}({\bf{\hat{n}}})\right|.
\end{align*}
By using the equation (\ref{fluxcombbrk}), the above equation reduces to
\begin{align*}
\|\Delta {\bf{J}}^{\text{brk}}({\bf{n}})-\Delta {\bf{J}}^{\text{brk}}({\bf{\hat{n}}})\| \leq \sum_{i=1}^I \left|\sum_{k=i+1}^{I} S_k(n_k-\hat{n}_k)\Delta x_k \Delta x_i b_{i,k} - S_i (n_i-\hat{n}_i) \sum_{j= 1}^{i-1} \frac{x_j}{x_i} b_{j,i} \Delta x_j \Delta x_i\right|.
\end{align*}
Since $x_j < x_i$ for $j< i$ and having $bS\leq Q_1$ from (\ref{kernelbound}), the above can be simplified as
\begin{align}\label{final}
\|\Delta {\bf{J}}^{\text{brk}}({\bf{n}})-\Delta {\bf{J}}^{\text{brk}}({\bf{\hat{n}}})\|& \leq {Q}_1 \sum_{i=1}^I \Delta x_i \sum_{k=1}^{I} \Delta x_k \left|n_k-\hat{n}_k\right| + {Q}_1 \sum_{i=1}^I \left|n_i-\hat{n}_i\right| \Delta x_i \sum_{j=1}^{I} \Delta x_j.
\end{align}
Therefore, the following is obtained
\begin{align*}
\|\Delta {\bf{J}}^{\text{brk}}({\bf{n}})-\Delta {\bf{J}}^{\text{brk}}({\bf{\hat{n}}})\| \leq 2 {Q}_1 x_{\text{max}} \|{\bf{{n}}}-{\bf{\hat{n}}}\|,
\end{align*}
with a Lipschitz constant $L_2= 2 {Q}_1 x_{\text{max}}$. Hence, the Lipschitz conditions for ${\bf{J(\hat{n})}}$ with 
a Lipschitz constant $L= (4C+6) Q C_{T,x_{\text{max}}}+ 2 {Q}_1 x_{\text{max}}$ is shown.
\end{proof}
Hence, by Theorem \ref{aaj2} the order of convergence of the FVS for the aggregation or 
breakage or coupled processes is same as the order of consistency which we have seen before in Lemma 
\ref{conserror}.

\section{Numerical Results}\label{num:resultsfvsaggbrk}
The mathematical results on convergence analysis are verified numerically for pure aggregation,
breakage and also for the combined processes considering several test problems. All numerical simulations below were carried out to 
investigate the experimental order of convergence (EOC) on four different types of meshes discussed in the next 
subsection. \enter

If the problem has analytical solutions, the following formula is used to calculate the EOC
\begin{align}\label{E:EOC_PureBreakage}
\text{EOC}= \ln(E_I/E_{2I})/\ln(2) \pp
\end{align}
Here $E_I$ and $E_{2I}$ are the discrete relative error norms calculated by dividing the error $\|N-\hat{N}\|$ by $\|N\|$ where $N, \hat{N}$ are the number of particles obtained mathematically and numerically, respectively. The symbols $I$ and $2I$ correspond to the number of degrees of freedom.\enter

Now, in case of unavailability of the analytical solutions, the EOC can be computed as
\begin{align}\label{E:EOC_GeneralBreakage}
\text{EOC}= \ln\bigg(\frac{\|\hat{N}_I-\hat{N}_{2I}\|}{\|\hat{N}_{2I}-\hat{N}_{4I}\|} \bigg)/\ln(2) \pk
\end{align}
where $\hat{N}_I$ is obtained by the numerical scheme using a mesh with $I$ degrees of freedom. \enter

Before going into the details of the test cases, in the following subsection we discuss briefly four different types 
of uniform and non-uniform meshes where global truncation errors are obtained numerically. These meshes have also been 
used in J. Kumar and Warnecke \cite{J_KUMARNM1:2008}.

\subsection{Meshes}

{\em{Uniform mesh}}: A uniform mesh is obtained when $\Delta x_i = \Delta x$ for all $i$.\\

{\em{Non-uniform smooth mesh}}: We are familiar with such a mesh from the previous section and 
Figure \ref{p:SmoothMesh}. For the numerical computations, a geometric mesh is considered.\\

% {\em{Locally uniform mesh}}\\
% 
% An example of a locally uniform mesh is considered in Figure \ref{p:locallySmoothMesh}. Let us consider that the computation domain is divided into finitely many sub-domains and each sub-domain is divided into an equal size mesh. In this way we get a locally uniform mesh. In our numerical simulation a geometric mesh is taken initially with 30 mesh points and for further level of computations each cell is divided into two equal parts.\\
% \begin{figure}
%  \centering
%  \input{LocallyUniformMesh}
%  \caption{Locally uniform smooth mesh.}\label{p:locallySmoothMesh}
% \end{figure}

{\em{Oscillatory mesh}}: The numerical verification has been done on an oscillatory mesh by taking 
$r=2$ in the equation (\ref{oscimesh}). In this case, the EOC is evaluated numerically by dividing the 
computation domain into 30 uniform mesh points initially. Then each cell is divided by a 1:2 ratio on further levels of computation.\\

{\em{Random mesh}}: Similar to the previous case, we started again with a geometric mesh with 30 grid 
points but then each cell is divided into two parts of random width in the further refined levels of computation. Here, we performed ten runs on different random grids and the relative errors are measured. The average of these errors over ten runs is used to calculate the EOC.

\subsection{Numerical examples}

\subsubsection{Pure aggregation}
{\bf{Test case 1:}}\\

The numerical verification of the EOC of the FVS for aggregation is discussed by taking two problems, namely the case of sum and product aggregation kernels. The analytical solutions for
both problems taking the negative exponential $n(0,x)= \exp(-\alpha x)$ as initial condition has been given in 
Scott \cite{Scott:1968}. Hence, the EOC is computed by using the relation (\ref{E:EOC_PureBreakage}). Table 
\ref{T:RelativeErrorAgg} shows that the EOC is 2 on uniform and non-uniform smooth meshes and is 1 on oscillatory and 
random grids in both cases. The computational domain in this case is taken as $[1E-6, 1000]$ which 
corresponds to the $\xi$ domain $[\ln(1E-6),\ln(1000)]$ for the exponential transformation 
$x=\exp(\xi)$ for the geometric mesh. The parameter $\alpha= 10$ was taken in the initial condition. 
The simulation result is presented at time $t= 0.5$ and $t= 0.3$ respectively for the sum and the 
product aggregation kernels corresponding to the aggregation extent 
$\hat{N}(t)/\hat{N}(0)\approx 0.80$.

\begin{table}
\caption{EOC (\ref{E:EOC_PureBreakage}) of the numerical schemes for {\bf{Test case 1}}.}
  \label{T:RelativeErrorAgg}
  \begin{center}\subtable[Uniform mesh]
    {\label{T:RelativeErrora} \begin{tabular}{l l l l l}
\\[-2ex]\toprule
Grid & $\beta(x,y)=x+y$ & $\beta(x,y)=xy$ \\points & Error\ \ \ \quad EOC &  Error\ \ \quad EOC\\
 \midrule
60  & 0.24E-3 \quad   -  &  0.0177 \quad  -  \\
120 & 0.11E-3 \quad 1.95 &  0.0045 \quad 1.96\\
240 & 0.04E-3 \quad 1.93 &  0.0012 \quad 1.94\\
480 & 0.01E-3 \quad 1.94 &  0.0003 \quad 1.92\\
\bottomrule
\end{tabular}}\hspace{1cm}
    \subtable[Non-uniform smooth mesh]
    {\label{T:RelativeErrora} \begin{tabular}{l l l l}
\\[-2ex]\toprule
Grid & $\beta(x,y)=x+y$ & $\beta(x,y)=xy$ \\points &  Error\ \quad EOC &  Error\ \quad EOC\\
 \midrule
 60  & 0.0047 \quad - &  0.0086 \quad -\\
 120  & 0.0012 \quad 1.99 &  0.0023 \quad 1.90\\
240   & 0.0003 \quad 1.98 &  0.0006 \quad 1.96\\
480   & 0.0001 \quad 2.00 &  0.0001 \quad 1.99\\
\bottomrule
\end{tabular}} \hspace{1cm}
\subtable[Oscillatory mesh]
    {\label{T:numericalerrorb} \begin{tabular}{l l l l l}
\\[-2ex]\toprule
Grid & $\beta(x,y)=x+y$ & $\beta(x,y)=xy$ \\points & Error\ \ \ \quad EOC &  Error\ \ \ \quad EOC\\
 \midrule
60  & 0.0029 \, \quad - & 0.0048 \, \quad -\\
120  & 0.0014 \, \quad 1.01 & 0.0019 \, \quad 1.29\\
240   & 6.05E-4 \quad 1.24 & 7.66E-4 \quad 1.31\\
480   & 2.20E-4 \quad 1.31 & 3.52E-4 \quad 1.12\\
\bottomrule
\end{tabular}}\hspace{1cm}
    \subtable[Random mesh]
    {\label{T:numericalerrorb} \begin{tabular}{l l l l}
\\[-2ex]\toprule
Grid & $\beta(x,y)=x+y$ & $\beta(x,y)=xy$ \\points & Error\ \quad EOC &  Error\ \quad EOC\\
 \midrule
 60  & 0.79E-3 \quad - &  0.0017 \quad -\\
 120  & 0.42E-3 \quad 0.98 &  8.2E-4 \quad 1.06\\
240   & 0.22E-3 \quad 1.02 &   2.8E-4 \quad 1.21\\
480   & 0.82E-4 \quad 1.21 &   1.5E-4 \quad 1.02\\
\bottomrule
\end{tabular}} 
\end{center}
\end{table}

\subsubsection{Pure breakage}

{\bf{Test case 2:}}\\

Here, the EOC is calculated for the binary breakage $b(x, y) = 2/y$ together with the linear and quadratic 
selection functions, i.e.\ $S(x) = x$ and $S(x) = x^2$. The analytical solutions for such problems
have been given in Ziff and McGrady \cite{Ziff:1985} for a mono-disperse initial condition of size 
unity, i.e.\ $n(0,x)= \delta(x-1)$. Hence, by using the formula (\ref{E:EOC_PureBreakage}),
we observe from the Table \ref{T:RelativeError} that the FVS is second order convergent on all the 
grids. The computational domain in this case is taken as $[1E-3, 1]$. Since the rate of breaking 
particles taking quadratic selection function is less than that of linear selection function, we 
take $t=100$, $200$ for linear and quadratic selection functions, respectively. The time has been 
chosen differently for both the selection functions to have the same extent of breakage 
$\hat{N}(t)/\hat{N}(0) \approx 22$. \\

{\bf{Test case 3:}}\\

Now the case of multiple breakage with the quadratic selection function $S(x) = x^2$ is considered where an analytical 
solution is not known. Therefore, the EOC is calculated using (\ref{E:EOC_GeneralBreakage}). 
For the numerical simulations, the following normal distribution as an initial condition is taken
\begin{align}\label{normalIC}
n(0,x)= \frac{1}{\sigma \sqrt{2\pi}}\exp\left(-\frac{(x-\mu)^2}{2\sigma^2}\right).
\end{align}
The computations are made for two breakage functions considered by Diemer and Olson \cite{Diemer:2002a} and Ziff \cite{Ziff:1991}, respectively
\begin{itemize}
\item case(i): $\dps b(x, y) = \frac{px^c(y-x)^{c+(c+1)(p-2)}[c+(c+1)(p-1)]!}{y^{pc+p-1}c![c+(c+1)(p-2)]!}, \quad p\in \mathbb{N}, p\geq 2$
\item case(ii): $\dps b(x, y) = \frac{12 x}{y^2}\left(1-\frac{x}{y}\right)$.
\end{itemize}
In case(i) the relation $\int_0^y b(x,y)dx= p$ 
holds where $p$ gives the total number of fragments per breakage event. The parameter $c\geq 0$ is 
responsible for the shape of the daughter particle distribution, see also \cite{Sommer:2006}. 
The numerical solutions are obtained using $p=4, c=2$. The second breakage function gives ternary 
breakage. For the numerical simulation the minimum and maximum values of $x$ are taken as $1E-3$ and 
$1$ respectively. The time $t=100$ is set to get the breakage extent 
$\hat{N}(t)/\hat{N}(0) \approx 22$ in case(i) while $t=150$ is used for case(ii). 
As expected from the mathematical analysis, we again observe from the Table \ref{Testcase2:RelativeError}
that the FVS shows convergence of second order on all the meshes. The computations for higher values 
of $p$ up to 19 are also tested and observed that there is no marked difference in the EOC.

\begin{table}
\caption{EOC (\ref{E:EOC_PureBreakage}) of the numerical schemes for {\bf{Test case 2}}.}
  \label{T:RelativeError}
  \begin{center}\subtable[Uniform smooth mesh]
    {\label{T:RelativeErrora} \begin{tabular}{l l l l}
\\[-2ex]\toprule
Grid & \ \ $S(x)=x$ & \ \ $S(x)=x^2$ \\points & Error\ \quad EOC &  Error\ \quad EOC\\
 \midrule
 60  & 0.3312 \quad - &  0.1870 \quad -\\
 120  & 0.0829 \quad 1.99 &  0.0482 \quad 1.95\\
240   & 0.0207 \quad 2.00 &  0.0126 \quad 1.94\\
480   & 0.0052 \quad 2.00 &  0.0034 \quad 1.90\\
\bottomrule
\end{tabular}} \hspace{1cm}
    \subtable[Non-uniform smooth mesh]
    {\label{T:RelativeErrora} \begin{tabular}{l l l l}
\\[-2ex]\toprule
Grid & \ \ $S(x)=x$ & \ \ $S(x)=x^2$ \\points & Error\ \quad EOC &  Error\ \quad EOC\\
 \midrule
 60  & 0.0526 \quad - &  0.1638 \quad -\\
 120  & 0.0136 \quad 1.95 &  0.0423 \quad 1.95\\
240   & 0.0034 \quad 1.99 &  0.0112 \quad 1.92\\
480   & 0.0009 \quad 2.00 &  0.0031 \quad 1.85\\
\bottomrule
\end{tabular}}\hspace{1cm}
\subtable[Oscillatory mesh]
    {\label{T:numericalerrorb} \begin{tabular}{l l l l}
\\[-2ex]\toprule
Grid & \ \ $S(x)=x$ & \ \ $S(x)=x^2$ \\points & Error\ \quad EOC &  Error\ \quad EOC\\
 \midrule
 60  & 0.0577 \quad - &  0.1310 \quad -\\
 120  & 0.0157 \quad 1.88 &  0.0376 \quad 1.80\\
240   & 0.0042 \quad 1.91 &   0.0105 \quad 1.84\\
480   & 0.0011 \quad 1.91 &   0.0030 \quad 1.82\\
\bottomrule
\end{tabular}}\hspace{1cm}
\subtable[Random mesh]
    {\label{T:numericalerrorb} \begin{tabular}{l l l l}
\\[-2ex]\toprule
Grid & \ \ $S(x)=x$ & \ \ $S(x)=x^2$ \\points & Error\ \quad EOC &  Error\ \quad EOC\\
 \midrule
 60  & 0.3516 \quad - &  1.1106 \quad -\\
 120  & 0.1001 \quad 1.81 &  0.3301 \quad 1.75\\
240   & 0.0282 \quad 1.83 &   0.0944 \quad 1.81\\
480   & 0.0078 \quad 1.85 &   0.0268 \quad 1.82\\
\bottomrule
\end{tabular}}
\end{center}
\end{table}

\begin{table}
\caption{EOC (\ref{E:EOC_GeneralBreakage}) of the numerical schemes for {\bf{Test case 3}}.}
  \label{Testcase2:RelativeError}
  \begin{center}
    \subtable[Uniform smooth mesh]
    {\label{T:RelativeErrora} \begin{tabular}{l l l l l}
\\[-2ex]\toprule
Grid & \quad case(i) & \quad case(ii) \\points & Error\ \quad EOC &  Error\ \quad EOC\\
 \midrule
 60  & - \quad \quad\quad \ - & - \quad\quad\quad \ - \\
 120  & 2.0655 \quad - & 4.7916 \quad -\\
240   & 0.6548 \quad 1.75 & 2.5829 \quad 2.16\\
480   & 0.1789 \quad 1.93 & 0.4364 \quad 1.91\\
960   & 0.0441 \quad 2.10 & 0.1792 \quad 1.67\\
\bottomrule
\end{tabular}} \hspace{1cm}
    \subtable[Non-uniform smooth mesh]
    {\label{T:numericalerrorb} \begin{tabular}{l l l l l}
\\[-2ex]\toprule
Grid & \quad case(i) & \quad case(ii) \\points & Error\ \quad EOC &  Error\ \quad EOC\\
 \midrule
 60  & - \quad \quad\quad \ - & - \quad \quad\quad \ - \\
 120  & 0.0244 \quad - & 0.0113 \quad -\\
240   & 0.0060 \quad 2.02 & 0.0028 \quad 2.01\\
480   & 0.0015 \quad 1.98 & 0.0007 \quad 2.00\\
960   & 0.0004 \quad 2.02 & 0.0002 \quad 2.00\\
\bottomrule
\end{tabular}} \hspace{1cm}
      \subtable[Oscillatory mesh]
   {\label{T:numericalerrorb} \begin{tabular}{l l l l l}
\\[-2ex]\toprule
Grid & \quad case(i) & \quad case(ii) \\points &  Error\ \quad EOC &  Error\ \quad EOC\\
\midrule
60  & - \quad \quad\quad \ - & - \quad \quad\quad \ - \\
120  & 0.78E-3 \quad - & 0.91E-3 \quad -\\
240   & 0.21E-3 \quad 1.74 & 0.28E-3 \quad 1.84\\
480   & 0.06E-3 \quad 1.93 & 0.09E-3 \quad 1.92\\
960   & 0.01E-3 \quad 2.02 & 0.02E-3 \quad 1.95\\
\bottomrule
\end{tabular}}\hspace{1cm}
\subtable[Random mesh]
    {\label{T:numericalerrorb} \begin{tabular}{l l l l l}
\\[-2ex]\toprule
Grid & \quad case(i) & \quad case(ii) \\points &  Error\ \quad EOC &  Error\ \quad EOC\\
 \midrule
 60  & - \quad \quad\quad \ - & - \quad \quad\quad \ - \\
 120  & 0.92E-3 \quad - & 0.89E-3 \quad -\\
240   & 0.18E-3 \quad 1.71 & 0.14E-3 \quad 1.82\\
480   & 0.05E-3 \quad 1.82 & 0.02E-3 \quad 1.90\\
960   & 0.02E-3 \quad 1.91 & 0.01E-3 \quad 1.92\\
\bottomrule
\end{tabular}}
\end{center}
\end{table}

\subsubsection{Coupled aggregation-breakage}
{\bf{Test case 4:}}\\

Finally, the EOC is evaluated for the simultaneous aggregation-breakage problem considering a constant 
aggregation kernel $\beta(x,y)= \beta_0$ and breakage kinetics $b(x,y)=2/y, S(x)=x$. The analytical solutions for this problem are given by Lage \cite{Lage:2002} for the following 
two different initial conditions
\begin{itemize}
\item case(i): $\dps n(0,x)= N_0\left[\frac{2N_0}{x_0}\right]^2 x \exp\left(-2 x \frac{N_0}{x_0}\right)$
\item case(ii): $\dps n(0,x)= N_0 \left[\frac{N_0}{x_0}\right] \exp\left(- x \frac{N_0}{x_0}\right).$
\end{itemize}
This is a special case where the number of particles stays constant. The later initial condition is a steady state 
solution. For the simulation the computational domain $[1E-2, 10]$ with $N_0= x_0= 1$ and time $t= 0.3$ is taken. 
From Table \ref{Testcase2:RelativeErrorAggBrk}, we find that the FVS is second order convergent on uniform and 
non-uniform smooth meshes and it gives first order on oscillatory and random meshes using 
(\ref{E:EOC_PureBreakage}). It should be mentioned that the computation has also been done for 
the product aggregation kernel $\beta(x,y)= xy$ and the linear selection function $S(x)=x$ taken 
together with two different general breakage functions as stated in the previous section. Analytical 
solutions are not available for such problems and so the EOC was calculated using 
(\ref{E:EOC_GeneralBreakage}). We observed again that the FVS shows similar results of convergence for 
these meshes.

\begin{table}
\caption{EOC (\ref{E:EOC_PureBreakage}) of the numerical schemes for {\bf{Test case 4}}.}
  \label{Testcase2:RelativeErrorAggBrk}
  \begin{center}
    \subtable[Uniform mesh]
    {\label{T:numericalerrorb} \begin{tabular}{l l l l}
\\[-2ex]\toprule
Grid & \quad case(i) & \quad case(ii) \\points &  Error\ \quad EOC &  Error\ \quad EOC\\
 \midrule
 60  & 0.3E-2 \quad    -  & 0.0032 \quad  -\\
 120  & 0.1E-2 \quad 1.75 & 0.0009 \quad 1.83\\
240   & 0.3E-3 \quad 1.86 & 2.4E-3 \quad 1.90\\
480   & 0.7E-4 \quad 2.01 & 0.7E-4 \quad 1.89\\
\bottomrule
\end{tabular}}\hspace{1cm}
    \subtable[Non-Uniform smooth mesh]
    {\label{T:RelativeErrora} \begin{tabular}{l l l l}
\\[-2ex]\toprule
Grid & \quad case(i) & \quad case(ii) \\points &  Error\ \quad EOC &  Error\ \quad EOC\\
 \midrule
60  & 0.0066 \quad   -  & 0.0018 \quad - \\
120 & 0.0018 \quad 1.90 & 0.0005 \quad 1.95\\
240 & 0.0004 \quad 1.97 & 0.0001 \quad 1.98\\
480 & 0.0001 \quad 2.00 & 2.9E-5 \quad 2.00\\
\bottomrule
\end{tabular}}\hspace{1cm}
   \subtable[Oscillatory mesh]
   {\label{T:numericalerrorb} \begin{tabular}{l l l l}
\\[-2ex]\toprule
Grid & \quad case(i) & \quad case(ii) \\points &  Error\ \quad EOC &  Error\ \quad EOC\\
\midrule
60  & 0.0019 \, \quad - & 0.0053 \, \quad - \\
120   & 0.62E-3 \quad 1.28 & 0.31E-2 \quad 0.98\\
240   & 0.29E-3 \quad 1.13 & 1.34E-3 \quad 1.07\\
480  & 0.15E-3 \quad 1.02 & 0.71E-3 \quad 1.06\\
\bottomrule
\end{tabular}} \hspace{1cm}
    \subtable[Random mesh]
    {\label{T:numericalerrorb} \begin{tabular}{l l l l}
\\[-2ex]\toprule
Grid & \quad case(i) & \quad case(ii) \\points & Error\ \quad EOC &  Error\ \quad EOC\\
 \midrule
60  & 0.0082 \, \quad -     & 0.0042 \, \quad - \\
120  & 0.0037 \, \quad 1.07 & 0.0023 \, \quad 0.91 \\
240   & 1.45E-3 \quad 1.22 & 0.0011 \, \quad 1.10\\
480   & 0.86E-3 \quad 1.01 & 0.04E-2 \quad 1.23\\
\bottomrule
\end{tabular}}
\end{center}
\end{table}

\section{Conclusions}\label{conclusions:fvsaggbrk}
In this article the convergence analysis of the finite volume techniques was studied for the non-linear aggregation 
and multiple breakage equations. We showed the consistency and then proved the Lipschitz continuity of the numerical 
fluxes to complete the convergence results. This investigation was based on the basic existing theorems and 
definitions from the book of Hundsdorfer and Verwer \cite{Hundsdorfer:2003} and the paper of Linz \cite{Linz:1975}. 
It was noticed that the scheme was second order convergent for a family of meshes for the pure breakage problem. 
For the aggregation and combined processes, it was not straightforward to evaluate the consistency and the 
convergence error on general meshes. This depended upon the type of grids chosen for the computations. Moreover, 
in these cases the method gave second order convergence on uniform and non-uniform smooth meshes while on 
non-uniform grids it showed only first order. The mathematical results of convergence 
analysis were verified numerically on several meshes by taking various examples of pure aggregation, pure breakage and 
the combined problems.

\section{Acknowledgements}
This work was supported by the DFG Graduiertenkollegs-828 and 1554, ({\em{Micro-Macro-Interactions in Structured Media and Particles Systems}})
Otto-von-Guericke-Universit\"{a}t Magdeburg. The authors gratefully
acknowledge for funding through this PhD programme.

\begin{appendix}
\section{Bound on total number of particles}\label{appendixFV_number}
We give the proof of Lemmas \ref{boundednessofnumbercontinuous} and \ref{boundednessofnumber} in Appendices \ref{conagbrk} and \ref{disagbrk}, respectively.
\subsection{Continuous aggregation and multiple breakage equation}\label{conagbrk}
\begin{proof}{[Lemma \ref{boundednessofnumbercontinuous}]}\\ Integrating the equation (\ref{E:FormCLtrun}) with respect to $x$ from $0$ to $x_{\text{max}}$ gives
\begin{align}\label{leibequationaagbrk}
\frac{d}{dt} \int_0^{x_{\text{max}}} n(t,x)dx= \int_0^{x_{\text{max}}} -\frac{1}{x} \frac{\partial}{\partial x} (F^\text{agg}+F^\text{brk}) dx.
\end{align}
From the equations (\ref{I:Aggfluxexacttrun}) and (\ref{I:Brkfluxexacttrun}), we know that
\begin{align*}
\frac{\partial}{\partial x} (F^\text{agg}(t,x)) = \frac{\partial}{\partial x} \int_0^x 
\int_{x-u}^{x_{\text{max}}} u \beta(u,v)n(t,u)n(t,v) dv du \quad \text{and}
\end{align*}
\begin{align*}
\frac{\partial}{\partial x} (F^\text{brk}(t,x)) = - \frac{\partial}{\partial x} \int_{x}^{x_{\text{max}}} \int_0^x u b(u,v) S(v) n(t,v) du dv.
\end{align*}
Applying the Leibniz integration rule on each of the flux separately ensures
\begin{align}\label{leibnizaggflux}
\frac{\partial}{\partial x} (F^\text{agg}(t,x)) = \int_0^{x_{\text{max}}} x \beta(x,v)n(t,x)n(t,v) 
dv -\int_0^x u \beta(u,x-u)n(t,u)n(t,x-u) du
\end{align}
and
\begin{align}\label{leibnizbrkflux}
\frac{\partial}{\partial x} (F^\text{brk}(t,x)) = - \int_{x}^{x_{\text{max}}} x b(x,v) S(v) n(t,v) dv +
\int_0^x u b(u,x) S(x) n(t,x) du.
\end{align}
Inserting (\ref{leibnizaggflux}) and (\ref{leibnizbrkflux}) into (\ref{leibequationaagbrk}) to get
\begin{align}\label{fubininumber}
\frac{d N(t)}{dt} =& \int_0^{x_{\text{max}}} \int_0^x \frac{u}{x} \beta(u,x-u)n(t,u)n(t,x-u) du dx -
 \int_0^{x_{\text{max}}} \int_0^{x_{\text{max}}} \beta(x,v)n(t,x)n(t,v) dv dx
 \nonumber \\&+ \int_{0}^{x_{\text{max}}} \int_{x}^{x_{\text{max}}} b(x,v) S(v) n(t,v) dv dx- \int_{0}^{x_{\text{max}}}
\int_0^x \frac{u}{x} b(u,x) S(x) n(t,x) du dx.
\end{align}
Changing the order of integration for the first and third integrals on the rhs of 
(\ref{fubininumber}) yields
\begin{align}
\frac{d N(t)}{dt} =& \int_0^{x_{\text{max}}} \int_u^{x_{\text{max}}} \frac{u}{x} \beta(u,x-u)n(t,u)n(t,x-u) dx du -
 \int_0^{x_{\text{max}}} \int_0^{x_{\text{max}}} \beta(x,v)n(t,x)n(t,v) dv dx
 \nonumber \\&+ \int_{0}^{x_{\text{max}}} \int_{0}^{v} b(x,v) S(v) n(t,v) dx dv- \int_{0}^{x_{\text{max}}}
\int_0^x \frac{u}{x} b(u,x) S(x) n(t,x) du dx.
\end{align}
Since $x\geq u$ for the first integral, this implies that $u/x \leq 1$. Substituting $x= z+u$ such 
that $dx= dz$, the above can be rewritten as
\begin{align*}
\frac{d N(t)}{dt} \leq & \int_0^{x_{\text{max}}} \int_0^{x_{\text{max}}-u} \beta(u,z)n(t,u)n(t,z) dz du -
 \int_0^{x_{\text{max}}} \int_0^{x_{\text{max}}} \beta(x,v)n(t,x)n(t,v) dv dx
 \nonumber \\&+ \int_{0}^{x_{\text{max}}} S(v) n(t,v) \int_{0}^{v} b(x,v) dx dv- \int_{0}^{x_{\text{max}}} \frac{S(x) n(t,x)}{x}
\int_0^x u b(u,x) du dx.
\end{align*}
Notice that the first two integrals combined give a negative value. Using the relation 
(\ref{breakageproperties}) of the breakage function in the last integral and due to negativity
\begin{align*}
\frac{d N(t)}{dt} \leq \int_{0}^{x_{\text{max}}} S(v) n(t,v) \int_{0}^{v} b(x,v) dx dv.
\end{align*}
From the bounds (\ref{kernelbound}) we know that $bS\leq Q_1$. Estimating $v\leq x_{\text{max}}$ leads to
\begin{align*}
\frac{d N(t)}{dt} \leq {Q}_1 x_{\text{max}} N(t).
\end{align*}
Therefore, the total number of particles is bounded and the bound is given as
\begin{align*}
{N}(t) \leq {N}(0) \exp(x_{\text{max}} Q_1 t) \leq {N}(0) \exp(x_{\text{max}} Q_1 T) = C_{T,x_{\text{max}}}.
\end{align*}
\end{proof}

\subsection{Discrete aggregation and multiple breakage equation}\label{disagbrk}

\begin{proof}{[Lemma \ref{boundednessofnumber}]}\\ Multiplying the equation (\ref{E:GenDis}) by 
$\Delta x_i/x_i$ and summing with respect to $i$ gives
\begin{align}\label{nobound1}
\frac{d (\sum_{i=1}^I \hat{n}_i(t)\Delta x_i)}{dt}  = -\sum_{i=1}^I \frac{1}{x_i}
\bigg[J^{\text{agg}}_{i+1/2} - J^{\text{agg}}_{i-1/2} + J^{\text{brk}}_{i+1/2} - J^{\text{brk}}_{i-1/2}\bigg].
\end{align}
We write out the summation over $i$ of the aggregation fluxes $J^{\text{agg}}_{i\pm 1/2}$ to get
\begin{align*}
-\sum_{i=1}^I \frac{1}{x_i} \bigg[J^{\text{agg}}_{i+1/2} - J^{\text{agg}}_{i-1/2}\bigg] = \frac{1}{x_1} J^{\text{agg}}_{1/2}- J^{\text{agg}}_{1+1/2}\left(\frac{1}{x_1}-\frac{1}{x_2}\right)&-\cdots- J^{\text{agg}}_{I- 1/2}\left(\frac{1}{x_{I-1}}-\frac{1}{x_I}\right)-\frac{1}{x_I} J^{\text{agg}}_{I+ 1/2}.
\end{align*}
For the breakage fluxes $J^{\text{brk}}_{i\pm 1/2}$ in (\ref{nobound1}) we substitute the definition 
(\ref{2}). Introducing the notations $\hat{N}_i(t)= \hat{n}_i(t) \Delta x_i$ and 
$\hat{N}(t)= \sum_{i=1}^I \hat{N}_i(t)$ ensure
\begin{align*}
\frac{d \hat{N}(t)}{dt}  =& \frac{1}{x_1}J^{\text{agg}}_{1/2} -\sum_{i=1}^{I-1} J^{\text{agg}}_{i+1/2}\left(\frac{1}{x_{i}}-\frac{1}{x_{i+1}}\right) - \frac{1}{x_I}J^{\text{agg}}_{I+1/2}\\
 & + \sum_{i=1}^I \sum_{k= i+1}^I \hat{N}_k(t) S(x_k) b(x_i,x_k) \Delta x_i - \sum_{i=1}^I \hat{N}_i(t) S(x_i) \sum_{j= 1}^{i-1} \frac{x_j}{x_i} b(x_j,x_i) \Delta x_j.
\end{align*}
Due to positivity of $J^{\text{agg}}_{i+1/2}$ for all $i$ and $J^{\text{agg}}_{1/2}= 0$, we estimate
\begin{align*}
\frac{d \hat{N}(t)}{dt} \leq \sum_{i=1}^I \sum_{k= i+1}^I \hat{N}_k(t) S(x_k) b(x_i,x_k) \Delta x_i - \sum_{i=1}^I \hat{N}_i(t) S(x_i) \sum_{j= 1}^{i-1} \frac{x_j}{x_i} b(x_j,x_i) \Delta x_j.
\end{align*}
Changing the order of summation for the first term and the summation indices in the second term yield
\begin{align*}
\frac{d \hat{N}(t)}{dt} \leq \sum_{k=1}^I \hat{N}_k(t) S(x_k) \left[\sum_{i= 1}^{k-1} b(x_i,x_k) \Delta x_i (1-x_i/x_k)\right].
\end{align*}
Since $i<k$ implies that $1-x_i/x_k < 1$. Having the bound $b S\leq Q_1$ gives 
$d \hat{N}(t)/dt \leq x_{\text{max}} {Q}_1 \hat{N}(t)$. 
% \begin{align*}
% \frac{d \hat{N}(t)}{dt} \leq & x_{\text{max}} {Q}_1 \hat{N}(t).
% \end{align*}
Therefore, the following bound is obtained on the total number of particles by using the FVS as
\begin{align*}
\hat{N}(t) \leq \hat{N}(0) \exp(x_{\text{max}} {Q}_1 t) \leq \hat{N}(0) \exp(x_{\text{max}} {Q}_1 T) = C_{T,x_{\text{max}}},
\end{align*}
which is the same bound as explained in the previous lemma, provided $\hat{N}(0)= {N}(0)$.
\end{proof}
\end{appendix}

\end{document}

%% file: UniformMesh.tex
%Created by jPicEdt 1.x
%Standard LaTeX format (emulated lines)
%Mon Jun 25 14:50:49 CEST 2007
\unitlength 1mm
\begin{picture}(100.00,36.25)(0,0)

\linethickness{0.15mm}
%Polygon 0 0(10.00,31.88)(90.00,31.88) 
\put(10.00,31.88){\line(1,0){80.00}}
%End Polygon

\linethickness{0.15mm}
%Polygon 0 0(90.00,31.88)(100.00,31.88) dash=1.00
\multiput(90.00,31.88)(1.82,0){6}{\line(1,0){0.91}}
%End Polygon

\linethickness{0.15mm}
%Polygon 0 0(10.00,33.13)(10.00,30.63) 
\put(10.00,30.63){\line(0,1){2.50}}
%End Polygon

\linethickness{0.15mm}
%Polygon 0 0(20.00,33.13)(20.00,30.63) 
\put(20.00,30.63){\line(0,1){2.50}}
%End Polygon

\linethickness{0.15mm}
%Polygon 0 0(30.00,33.13)(30.00,30.63) 
\put(30.00,30.63){\line(0,1){2.50}}
%End Polygon

\linethickness{0.15mm}
%Polygon 0 0(40.00,33.13)(40.00,30.63) 
\put(40.00,30.63){\line(0,1){2.50}}
%End Polygon

\linethickness{0.15mm}
%Polygon 0 0(50.00,33.13)(50.00,30.63) 
\put(50.00,30.63){\line(0,1){2.50}}
%End Polygon

\linethickness{0.15mm}
%Polygon 0 0(60.00,33.13)(60.00,30.63) 
\put(60.00,30.63){\line(0,1){2.50}}
%End Polygon

\linethickness{0.15mm}
%Polygon 0 0(70.00,33.13)(70.00,30.63) 
\put(70.00,30.63){\line(0,1){2.50}}
%End Polygon

\linethickness{0.15mm}
%Polygon 0 0(80.00,33.13)(80.00,30.63) 
\put(80.00,30.63){\line(0,1){2.50}}
%End Polygon

\linethickness{0.15mm}
%Polygon 0 0(48.75,27.50)(48.75,15.63)(47.50,15.63)(50.00,11.25)(52.50,15.63)(51.25,15.63)(51.25,27.50)(48.75,27.50) 
\put(48.75,15.63){\line(0,1){11.88}}
\put(47.50,15.63){\line(1,0){1.25}}
\multiput(47.50,15.63)(0.12,-0.21){21}{\line(0,-1){0.21}}
\multiput(50.00,11.25)(0.12,0.21){21}{\line(0,1){0.21}}
\put(51.25,15.63){\line(1,0){1.25}}
\put(51.25,15.63){\line(0,1){11.88}}
\put(48.75,27.50){\line(1,0){2.50}}
%End Polygon

\put(5.00,31.88){\makebox(0,0)[cc]{$\xi$}}

\put(80.00,11.88){\makebox(0,0)[cc]{}}

\linethickness{0.15mm}
%Polygon 0 0(10.00,6.25)(90.00,6.25) 
\put(10.00,6.25){\line(1,0){80.00}}
%End Polygon

\linethickness{0.15mm}
%Polygon 0 0(90.00,6.25)(100.00,6.25) dash=1.00
\multiput(90.00,6.25)(1.82,0){6}{\line(1,0){0.91}}
%End Polygon

\linethickness{0.15mm}
%Polygon 0 0(10.00,7.50)(10.00,5.00) 
\put(10.00,5.00){\line(0,1){2.50}}
%End Polygon

\linethickness{0.15mm}
%Polygon 0 0(13.13,7.50)(13.13,5.00) 
\put(13.13,5.00){\line(0,1){2.50}}
%End Polygon

\linethickness{0.15mm}
%Polygon 0 0(17.50,7.50)(17.50,5.00) 
\put(17.50,5.00){\line(0,1){2.50}}
%End Polygon

\linethickness{0.15mm}
%Polygon 0 0(25.63,7.50)(25.63,5.00) 
\put(25.63,5.00){\line(0,1){2.50}}
%End Polygon

\linethickness{0.15mm}
%Polygon 0 0(37.50,7.50)(37.50,5.00) 
\put(37.50,5.00){\line(0,1){2.50}}
%End Polygon

\linethickness{0.15mm}
%Polygon 0 0(50.00,7.50)(50.00,5.00) 
\put(50.00,5.00){\line(0,1){2.50}}
%End Polygon

\linethickness{0.15mm}
%Polygon 0 0(65.00,7.50)(65.00,5.00) 
\put(65.00,5.00){\line(0,1){2.50}}
%End Polygon

\linethickness{0.15mm}
%Polygon 0 0(88.75,7.50)(88.75,5.00) 
\put(88.75,5.00){\line(0,1){2.50}}
%End Polygon

\put(5.63,6.25){\makebox(0,0)[cc]{$x$}}

\put(28.13,21.88){\makebox(0,0)[cc]{smooth transformation}}

\put(27.50,16.25){\makebox(0,0)[cc]{$x = g(\xi)$}}

\put(5.63,31.25){\makebox(0,0)[cc]{}}

\put(50.00,36.25){\makebox(0,0)[cc]{uniform mesh}}

\put(50.00,2.50){\makebox(0,0)[cc]{non-uniform mesh}}

\put(73.13,21.88){\makebox(0,0)[cc]{for example}}

\put(73.13,17.50){\makebox(0,0)[cc]{$x = \exp(\xi)$}}

\end{picture}